\documentclass[11pt]{article}
 
\usepackage[margin=1in]{geometry} 
\usepackage{amsmath,amsthm,amssymb}
\usepackage{mathtools}
\usepackage{color}
\usepackage{lipsum} 
\usepackage{fancyhdr}
\usepackage{tikz}
\usepackage{tikz-cd}
\usetikzlibrary{patterns}
\usepackage{ytableau}
\usepackage{physics}
\usepackage{enumerate}
\usepackage{authblk}
\newcommand{\Z}{\mathbb{Z}}
\newcommand{\Q}{\mathbb{Q}}
\newcommand{\R}{\mathbb{R}}
\newcommand{\C}{\mathbb{C}}
\newcommand{\uq}{U_q'(\mathfrak{so}_3)}
\newcommand{\uqq}{U_q'(\mathfrak{so}_4)}
\newcommand{\uqqq}{U_q'(\mathfrak{so}_n)}

\newcommand{\inv}{^{-1}}

\newcommand{\m}{\mathcal{M}}
\newcommand{\jfin}{\mathcal{J}^{\text{fin}}(A)}
\newcommand{\uson}{U(\mathfrak{so}_n)}

\newtheorem{theorem}{Theorem}[section]
\newtheorem{lemma}[theorem]{Lemma}
\newtheorem*{claim*}{Claim}
\newtheorem{corollary}[theorem]{Corollary}
\newtheorem{proposition}[theorem]{Proposition}

\newtheorem{thmintro}{Theorem}

\theoremstyle{definition}
\newtheorem{definition}[theorem]{Definition}
\newtheorem{example}[theorem]{Example}

\begin{document}
\title{Generic Gelfand-Tsetlin Modules of Quantized and Classical Orthogonal Algebras} 
\author{Jordan Disch}
\affil{Department of Mathematics, Iowa State University}
\date{}
\maketitle

\begin{abstract}
We construct infinite-dimensional analogues of finite-dimensional simple modules of the nonstandard $q$-deformed enveloping algebra $\uqqq$ defined by Gavrilik and Klimyk, and we do the same for the classical universal enveloping algebra $\uson$. In this paper we only consider the case when $q$ is not a root of unity, and $q\to 1$ for the classical case. Extending work by Mazorchuk on $\mathfrak{so}_n$, we provide rational matrix coefficients for these infinite-dimensional modules of both $\uqqq$ and $\uson$. We use these modules with rationalized formulas to embed the respective algebras into skew group algebras of shift operators. Casimir elements of $\uqqq$ were given by Gavrilik and Iorgov, and we consider the commutative subalgebra $\Gamma\subset\uqqq$ generated by these elements and the corresponding subalgebra $\Gamma_1\subset\uson$. The images of $\Gamma$ and $\Gamma_1$ under their respective embeddings into skew group algebras are equal to invariant algebras under certain group actions. We use these facts to show that $\Gamma$ is a Harish-Chandra subalgebra of $\uqqq$ and $\Gamma_1$ is a Harish-Chandra subalgebra of $\uson$.
\end{abstract}

\begin{section}{Introduction} \label{intro}
In \cite{gav}, Gavrilik and Klimyk gave finite-dimensional simple modules of the algebra $\uqqq$, a $q$-deformation of the universal enveloping algebra $\uson$ which differs from the Drinfeld-Jimbo quantum group $U_q(\mathfrak{so}_n)$. In fact, $\uqqq$ is not a Hopf algebra, but rather a coideal subalgebra of the Drinfeld-Jimbo quantum group $U_q(\mathfrak{gl}_n)$, \cite{mol}. Moreover, $\uqqq$ is an example of a generalization of quantum groups called $\imath$quantum groups, \cite{wang}. $\uqqq$ and $U_q(\mathfrak{gl}_n)$ are realized as a quantum symmetric pair in \cite{nou}, \cite{let}. There are two crucial advantages of the algebra $\uqqq$ compared to the Drinfeld-Jimbo version. First, there exists a natural chain of subalgebras:
\begin{equation*}
    U_q'(\mathfrak{so}_2)\subset \uq \subset \cdots \subset \uqqq.
\end{equation*}\noindent Second, the algebra $\uqqq$ is defined for all nonzero complex numbers $q$, including $q=1$ when it becomes isomorphic to $\uson$.

The finite-dimensional simple modules of $\uqqq$ given in \cite{gav} correspond to finite-dimensional simple modules of $\uson$ when $q\to 1$, thus these modules are called \emph{classical}, \cite{nonclassical}. As is the case for $\uson$, the classical modules of $\uqqq$ are parameterized by combinatorial patterns called \emph{Gelfand-Tsetlin patterns} (GT patterns). These patterns are the same for finite-dimensional simple modules of $\uson$ and the classical modules of $\uqqq$, and they are given in \cite{gz}. Ueno et al. rationalized the formulas for finite-dimensional simple modules of $U_q(\mathfrak{gl}_n)$ in \cite{kim}, which Mazorchuk and Turowska used to construct infinite-dimensional analogues of finite-dimensional modules of $U_q(\mathfrak{gl}_n)$ in \cite{maz}, calling these modules \emph{generic Gelfand-Tsetlin modules} (GT modules). Generic GT modules were constructed by the author in \cite{disch} for $\uq$ and $\uqq$, and it was shown that the modules have finite length. (We do not discuss the length of the modules constructed in this paper, though we conjecture that they have finite length for general $n$.) The term ``generic" refers to the fact that the patterns that parameterize these modules are unrestricted from the integral and interlacing conditions on the entries required in GT patterns. A \emph{GT module} of an algebra $U$ is a module that is a direct sum of common generalized eigenspaces of a certain commutative subalgebra $\Gamma\subset U$, \cite{dfo}. Mazorchuk constructed generic GT modules of $\uson$ in \cite[Theorem~1]{ma}. We extend his work in this paper to include the quantum case, and fully rationalized matrix coefficients for both the quantum and classical cases. Following \cite{dfo} and \cite{maz}, we use rationalized matrix coefficients to prove the existence of generic GT modules. This brings us to the first main theorem of this paper. Let $N$ be the number of boxes in a GT pattern for $\mathfrak{so}_n$. As in \cite{maz}, we call a pattern \emph{admissible} if it belongs to the complement of a certain countable union of affine hyperplanes in $\C^N$ (see Section \ref{exist_gen_gz_mods} for details).
\begin{thmintro} \label{gen_gz_mods_intro}
    There exists a family of generic GT modules $V_{\alpha^0}^{m_n}$ of $\uqqq$ and $\uson$ parameterized by admissible patterns. Explicit bases and actions for $V_{\alpha^0}^{m_n}$ are given in Theorem \ref{existence_gen_n}. Furthermore, the intersection of annihilators of all these modules $V_{\alpha^0}^{m_n}$ is zero in both $\uqqq$ and $\uson$.
\end{thmintro}\noindent Note that in this paper, as in \cite{dfo} and \cite{maz}, the basis vectors of $V_{\alpha^0}^{m_n}$ are common eigenvectors of $\Gamma$ rather than merely being generalized eigenvectors. 

It was shown by Futorny and Ovsienko in \cite[Proposition~7.2]{fo10} that the universal enveloping algebra $U(\mathfrak{gl}_n)$ is isomorphic to a subalgebra of a skew group algebra. More precisely, they give an embedding from $U(\mathfrak{gl}_n)$ to the invariant subalgebra of a skew group algebra under the action of a certain group. This was done to show that $U(\mathfrak{gl}_n)$ is a Galois order. Similarly, Futorny and Hartwig showed that the quantum group $U_q(\mathfrak{gl}_n)$ embeds into a skew group algebra in \cite[Proposition~5.9]{fut_hartwig} and show that it is a Galois ring, and furthermore it was shown by Hartwig that $U_q(\mathfrak{gl}_n)$ is a Galois order in \cite{hart}. Though we do not prove in this paper that $\uqqq$ or $\uson$ are Galois orders nor Galois rings, we conjecture this, and we do embed them into skew group algebras. The images of these embeddings are not contained in the invariant subalgebra under our particular group action on the skew group algebra, but we do obtain a similar result for commutative subalgebras $\Gamma$ and $\Gamma_1$ of $\uqqq$ and $\uson$, respectively, generated by Casimir elements from \cite{cas}.

In more detail, let $\Lambda$ be a Laurent polynomial algebra with number of generators equal to the number of entries in a GT pattern for $\uqqq$, and $L:=\operatorname{Frac}\Lambda$. Likewise, let $\Lambda_1$ be a polynomial algebra with the same number of generators as in $\Lambda$, and $L_1:=\operatorname{Frac}\Lambda_1$. We define $\m$ to be a free abelian group with number of generators equal to the number of entries in a GT pattern with the top row deleted, and consider skew group algebras $L\#\m$ and $L_1\#\m$, where $\m$ acts on $L_1$ ($L$) by ($q$-)shift operators. We define the group $G=W_2^q\times\dots\times W_n^q$ where each $W_i^q$ is an extension of the Weyl group of the quantum group $U_q(\mathfrak{so}_i)$ \cite[Theorem~9.1.6]{chari} for $q\neq 1$, and each $W_i^1$ is the Weyl group of $\uson$. Explicitly,
\begin{equation*}
    W_i^q=\begin{cases}
    (C_2^q)^k\rtimes S_k, & i=2k+1\\
    (C_2^q)^k_{\text{even}}\rtimes S_k, & i=2k
    \end{cases},
\end{equation*}
\begin{equation*}
    C_2^q=\begin{cases}
    (\Z/2\Z)^2, & q\neq 1\\
    \Z/2\Z, & q=1
    \end{cases},
\end{equation*}and $(C_2^q)^k_{\text{even}}$ is the subgroup of $(C_2^q)^k$ consisting of an even number of sign changes. This brings us to the second main result of this paper.
\begin{thmintro} \label{embedding_intro}
Let $\Gamma$ be the subalgebra of $\uqqq$ generated by the Gavrilik-Iorgov Casimirs of $U_q'(\mathfrak{so}_i)$, $i=2,...,n$, and $\Gamma_1\subset\uson$ the corresponding subalgebra when $q\to 1$. The following are true:
    \begin{enumerate}[{\rm (i)}]
    \item \label{b_pt_i} There exist injective algebra homomorphisms $\varphi:\uqqq\hookrightarrow L\#\m$ and $\varphi_1:\uson\hookrightarrow L_1\#\m$.
    \item \label{b_pt_ii} $\varphi(\Gamma)=\Lambda^G$ and $\varphi_1(\Gamma_1)=\Lambda_1^G$.
    \end{enumerate}
\end{thmintro}

Drozd, Futorny, and Ovsienko defined the notion of a \emph{Harish-Chandra subalgebra} and from there established the study of Harish-Chandra categories in \cite{dfo}. For simplicity, the definition of a Harish-Chandra subalgebra we use in this paper is restricted to the case in which the subalgebra $C$ of $U$ is commutative. Thus we say a Harish-Chandra subalgebra $C\subset U$ is a subalgebra such that $CuC$ is finitely generated as a left and right $C$-module for all $u\in U$. This property is called ``quasi-central" in \cite{dfo}. Mazorchuk and Turowska show that the GT subalgebra of $U_q(\mathfrak{gl}_n)$ is a Harish-Chandra subalgebra in \cite[Proposition~1]{maz}. Mazorchuk also stated a similar result for $\uson$ in \cite[Remark~1]{ma}. Our final main theorem of this paper enables the future study of Harish-Chandra categories of $\uqqq$ and $\uson$.
\begin{thmintro} \label{hc_subalg_intro}
$\Gamma$ is a Harish-Chandra subalgebra of $\uqqq$, and $\Gamma_1$ is a Harish-Chandra subalgebra of $\uson$.
\end{thmintro}

In Section \ref{prelim}, we provide necessary background information, more convenient versions of the formulas from \cite{gav}, and technical lemmas. We prove Theorem \ref{gen_gz_mods_intro} in Section \ref{gen_gz_mods}, Theorem \ref{embedding_intro} \eqref{b_pt_i} in Section \ref{emb}, and Theorem \ref{embedding_intro} \eqref{b_pt_ii} and Theorem \ref{hc_subalg_intro} in Section \ref{hc}. Section \ref{appendix} is an appendix containing a proof of a technical lemma from Section \ref{gen_gz_mods}.

\end{section}

\begin{section}{Preliminaries} \label{prelim}
\begin{subsection}{$\uqqq$ and Finite-Dimensional Simple Modules} \label{fd_simple_mods}
Throughout this paper we fix $h\in\C\setminus 2\pi i\Q$ where $i=\sqrt{-1}$ and put $q=e^{h}$. This ensures that $q$ is not a root of unity. Here we define $\uqqq$ for $n\geq 2$ using the same presentation as \cite[Section~2]{gav}. $U_q'(\mathfrak{so}_2)$ is just the polynomial algebra in one generator $\C[I_{21}]$.

For all $b\in\C$, we define $q^b=e^{hb}$ and
\begin{equation*}
    [b]:=\dfrac{q^{b}-q^{-b}}{q-q^{-1}}.
\end{equation*}
$\uqqq$ is defined as the complex associative algebra generated by elements $I_{i,i-1}$, $i=2,...,n$, which satisfy the following relations:
\begin{subequations}
\begin{equation}\label{rel1}
    [I_{i,i-1},I_{j,j-1}]=0 \ \ \ \ \ \text{if} \ |i-j|>1,
    \end{equation}
    \begin{equation}\label{rel2}
    I_{i+1,i}^2I_{i,i-1}-[2]I_{i+1,i}I_{i,i-1}I_{i+1,i}+I_{i,i-1}I_{i+1,i}^2=-I_{i,i-1},
    \end{equation}
    \begin{equation}\label{rel3}
    I_{i,i-1}^2I_{i+1,i}-[2]I_{i,i-1}I_{i+1,i}I_{i,i-1}+I_{i+1,i}I_{i,i-1}^2=-I_{i+1,i}.
    \end{equation}
\end{subequations}
    As $q\to 1$, this becomes the universal enveloping algebra of the complex semisimple Lie algebra $\mathfrak{so}_n$. This is because \eqref{rel2} and \eqref{rel3} become
\begin{equation*}
        [I_{i+1,i},[I_{i+1,i},I_{i,i-1}]]=-I_{i,i-1},
    \end{equation*}
    \begin{equation*}
        [I_{i,i-1},[I_{i,i-1},I_{i+1,i}]]=-I_{i+1,i},
    \end{equation*}respectively.
    
The following definition is needed for proving Lemma \ref{zar}, and it is also needed to define finite-dimensional modules of $\uqqq$.
\begin{definition}
For $n\geq 2$, consider a $k$-tuple $(m_{n1},m_{n2},...,m_{nk})\in\Z^k\cup(\frac{1}{2}+\Z)^k$ where $k=\operatorname{rank}\mathfrak{so}_n=\lfloor\frac{n}{2}\rfloor$ and
\begin{subequations}
    \begin{align}
        m_{n1}\geq m_{n2} \geq \dots \geq m_{nk}\geq 0, &  \ \text{for $n$ odd} \label{highest_wt_odd}\\
        m_{n1}\geq m_{n2} \geq \dots \geq |m_{nk}|, &  \ \text{for $n$ even}. \label{highest_wt_even}
    \end{align}
    \end{subequations}
    A \emph{Gelfand-Tsetlin pattern} (GT pattern) for $U_q'(\mathfrak{so}_n)$ has the following form: 
\[
\begin{matrix}
m_{n1} & \dots & \dots & \dots & m_{nk}\\
 & & \vdots & &\\
 & m_{51} & & m_{52} &\\
 & m_{41} & & m_{42} &\\
  & & m_{31} & & \\
 & & m_{21} & &
\end{matrix}
\]where the entries are all integers or all half-integers and satisfy the following interlacing conditions from \cite{gz}: 
\begin{subequations} \label{interlacing_conditions}
\begin{equation} \label{interlacing_odd}
    m_{2p+1,1}\geq m_{2p,1}\geq m_{2p+1,2}\geq m_{2p,2} \geq \dots \geq m_{2p+1,p}\geq m_{2p,p}\geq -m_{2p+1,p},
\end{equation}
\begin{equation} \label{interlacing_even}
    m_{2p,1}\geq m_{2p-1,1}\geq m_{2p,2}\geq m_{2p-1,2}\geq \dots \geq m_{2p-1,p-1}\geq |m_{2p,p}|.
\end{equation}
\end{subequations}
\end{definition}
The top row is the highest weight of a module, while the lower entries parameterize the basis vectors in the following finite-dimensional modules from \cite[Section~3]{gav}, on which our generic Gelfand-Tsetlin modules are based. We call the following modules \emph{classical} \cite{nonclassical} since they correspond to modules of $\uson$ when $q\to 1$. If $\alpha$ is a GT pattern which labels a basis element of our module, 
\begin{subequations} \label{i2p_fd_action}
\begin{equation} \label{i2p1}
    I_{2p+1,2p}.\ket{\alpha}=\sum_{j=1}^pA_{2p}^j(\alpha)\ket{\alpha^{+j}_{2p}} - \sum_{j=1}^pA_{2p}^j({\alpha}^{-j}_{2p})\ket{{\alpha}^{-j}_{2p}}
\end{equation}
\begin{equation} \label{i2p2}
    I_{2p,2p-1}.\ket{\alpha}=\sum_{j=1}^{p-1}B_{2p-1}^j(\alpha)\ket{\alpha^{+j}_{2p-1}} - \sum_{j=1}^{p-1}B_{2p-1}^j({\alpha}^{-j}_{2p-1})\ket{{\alpha}^{-j}_{2p-1}}+i\cdot C_{2p-1}(\alpha)\ket{\alpha}
\end{equation}
\end{subequations}
where $\alpha^{\pm j}_i$ is obtained from $\alpha$ by replacing $m_{ij}$ by $m_{ij}\pm 1$; $l_{2p,j}=m_{2p,j}+p-j$, $l_{2p+1,j}=m_{2p+1,j}+p-j+1$ and
\begin{subequations} \label{coeffsnoabs}
\begin{multline}\label{a2pj}
    A_{2p}^j(\alpha):=\bigg(\dfrac{[l_j'][l_j'+1]}{[2l_j'][2l_j'+2]}\dfrac{\prod_{r=1}^j[l_r+l_j'][l_r-l_j'-1]\prod_{r=j+1}^{p}[l_j'+l_r][l_j'-l_r+1]}{\prod_{r< j}[l_r'+l_j'][l_r'-l_j'][l_r'+l_j'+1][l_r'-l_j'-1]}\\
    \cdot \dfrac{\prod_{r=1}^{j-1}[l_r''+l_j'][l_r''-l_j'-1]\prod_{r=j}^{p-1}[l_j'+l_r''][l_j'-l_r''+1]}{\prod_{r>j}[l_j'+l_r'][l_j'-l_r'][l_j'+l_r'+1][l_j'-l_r'+1]}\bigg)^{1/2}
\end{multline}
\begin{multline}\label{b2pj}
    B_{2p-1}^j(\alpha):=\bigg(\dfrac{\prod_{r=1}^{j}[l_r'+l_j''][l_r'-l_j'']\prod_{r=j+1}^{p}[l_j''+l_r'][l_j''-l_r']}{\prod_{r< j}[l_r''+l_j''][l_r''-l_j''][l_r''+l_j''-1][l_r''-l_j''-1]}\\
    \cdot\dfrac{\prod_{r=1}^{j-1}[l_r'''+l_j''][l_r'''-l_j'']\prod_{r=j}^{p-1}[l_j''+l_r'''][l_j''-l_r''']}{\prod_{r>j}[l_j''+l_r''][l_j''-l_r''][l_j''+l_r''-1][l_j''-l_r''+1]}\dfrac{1}{[l_j'']^2[2l_j''+1][2l_j''-1]}\bigg)^{1/2}
\end{multline}
\begin{equation}
    C_{2p-1}(\alpha):=\frac{\prod_{r=1}^{p}[l_{r}']\prod_{r=1}^{p-1}[l_{r}''']}{\prod_{r=1}^{p-1}[l_{r}''][l_{r}''-1]};
\end{equation}
\end{subequations}
where $l_{2p+1,i}$ is denoted by $l_i$, $l_{2p,i}$ by $l_i'$, $l_{2p-1,i}$ by $l_i''$, and $l_{2p-2,i}$ by $l_i'''$. These are expanded versions of the matrix coefficients given in \cite[Section~3]{gav}, namely we rewrite the coefficients more explicitly without absolute values. It is convenient for the $q$-numbers to be positive in the proof of Proposition \ref{rat}, which is why we make this change. 
\end{subsection}

\begin{subsection}{$q$-Numbers}
We collect some elementary properties of $q$-numbers that will be useful. We use the fact that $[b]=0$ if and only if $q^{2b}=1$.
\begin{lemma}\label{lemq}
The following are true:
\begin{enumerate}[{\rm (i)}]
    \item\label{lemq1}If $b\in \Q\setminus\{0\}$, then $[b]\neq 0$.
    \item\label{lemq1.25}If $a,b\in\Q$ and $a\neq b$, then $q^a\neq q^b$.
    \item\label{lemq3}$\frac{[m]}{[2m]}$ exists if and only if $q^{2m}\neq -1$.
\end{enumerate}
\end{lemma}
\begin{proof}
Items \eqref{lemq1} and \eqref{lemq3} were proven in \cite[Lemma~2.1]{disch}. We prove item \eqref{lemq1.25} by contradiction. Suppose $q^a=q^b$ where $a,b\in\Q$ and $a\neq b$. $q^{a-b}=1$ for nonzero rational number $a-b=:\frac{c}{d}$ where $c,d\in\Z$. Then $q^c=1^d=1$ where $c\in \Z\setminus\{0\}$, which contradicts the fact that $q$ is not a root of unity.
\end{proof}
\end{subsection}

\begin{subsection}{Zariski-Density}
The idea of Zariski-density of GT patterns is well-known and has been employed for $U(\mathfrak{gl}_n)$ \cite{dfo}; $U_q(\mathfrak{gl}_n)$ \cite{fut_hartwig}, \cite{maz}; and $\uson$ \cite{ma}. For completeness we formulate and prove precise statements that will be crucial in the proofs of the main results of the paper.

Define $s:=r_2+\dots+r_n$, where $r_i=\operatorname{rank}\mathfrak{so}_i=\lfloor\frac{i}{2}\rfloor$. Thus $s=k^2+k$ if $n$ is odd and $s=k^2$ if $n$ is even. 
\begin{lemma} \label{zar}
Consider a rational function $f\in \C(x_1,...,x_s)$. 
\begin{enumerate}[{\rm (i)}]
    \item\label{zar4} If $f$ is zero at all $q^{m_{ij}}$ for any GT pattern for $U_q'(\mathfrak{so}_n)$, then $f=0$. 
    \item\label{zar4_classic} If $f$ is zero at all $m_{ij}$ for any GT pattern for $U_q'(\mathfrak{so}_n)$, then $f=0$. 
    \item\label{zar_generic} If $f$ is zero at all $q^{a_{j}}$ for $a_j=\frac{1}{p_j}$, $1\leq j\leq s$, for any strictly decreasing sequence of odd primes $(p_j)_{j=1}^{s}$, then $f=0$. 
    \item\label{zar_generic_classic} If $f$ is zero at all ${a_{j}}$ for $a_j=\frac{1}{p_j}$, $1\leq j\leq s$, for any strictly decreasing sequence of odd primes $(p_j)_{j=1}^{s}$, then $f=0$. 
\end{enumerate}
\end{lemma}
\begin{proof}
Item \eqref{zar4} is proved by the author in \cite[Lemma~2.4]{disch}. The proof for item \eqref{zar4_classic} is identical, except we replace $q^a$ with $a$ for any $a\in\frac{1}{2}\Z$. We prove item \eqref{zar_generic} here, though it is similar to the previous proofs.

Let $S:=\{(1/p_1,...,1/p_s)\in\C^s\mid (p_j)_{j=1}^s \; \text{is a strictly decreasing sequence of odd primes}\}$. Let $S_q:=\{(q^{1/p_1},...,q^{1/p_s})\in\C^s\mid (1/p_1,...,1/p_s)\in S\}$. We want to show that $S_q$ is Zariski-dense, i.e. if a complex rational function $f$ is zero when evaluated at $S_q$, then $f=0$. Now suppose $f(q^{a_1},q^{a_2},...,q^{a_s})=0$ for all $\textbf{a}\in S$. We may assume $f\in \C[x_1,...,x_s]$ since $f=0$ if and only if the numerator is $0$.

There is a nonnegative integer $N_0$ where for $0\leq k\leq N_0$ there exists $c_k^{(1)}(x_2,x_3,...,x_s)\in \C[x_2,x_3,...,x_s]$ such that 
\begin{equation*}
    f(x_1,...,x_s)=\sum_{k=0}^{N_0}c_k^{(1)}(x_2,...,x_s)x_1^k.
\end{equation*}Fix allowed $a_2,...,a_s$. Then if $\textbf{v}_1:=(q^{a_2},...,q^{a_s})$, we have 
\begin{equation*}
    g_{\textbf{v}_1}(x_1):=f(x_1,q^{a_2},...,q^{a_s})\in \C[x_1].
\end{equation*}Note that $\{q^{a_1}\mid a_1=\frac{1}{p_1}, \ p_1>p_2\}$ is an infinite set of zeros for the single-variable polynomial $g_{\textbf{v}_1}(x_1)$ since $p_2$ is a fixed prime number, so $g_{\textbf{v}_1}=0$ in $\C[x_1]$. Since $\textbf{v}_1$ was arbitrary, this implies $c_k^{(1)}(q^{a_2},...,q^{a_s})=0$ for all $0\leq k\leq N_0$ and for all $\textbf{a}\in S_n$. If we show that $c_k^{(1)}=0$ for all $0\leq k\leq N_0$, then we are done. We iterate this process to obtain the desired result.

The proof of item \eqref{zar_generic_classic} is identical to the proof of item \eqref{zar_generic}, except we replace $q^a$ with $a$ for any $a\in\Q$, and it is unnecessary to invoke Lemma \ref{lemq} \eqref{lemq1.25}.
\end{proof}
\end{subsection}
\end{section}

\begin{section}{Generic Gelfand-Tsetlin Modules} \label{gen_gz_mods}
In Section \ref{rational} we provide re-scalings of the classical finite-dimensional simple modules of $\uqqq$ which yield rational matrix coefficients in the variables $q^{m_{ij}}$, where $m_{ij}$ denotes an entry in a GT pattern. We use this to construct generic GT modules of $\uqqq$ and $\uson$ in Section \ref{exist_gen_gz_mods}, and show that the intersection of annihilators is trivial in both algebras in Section \ref{triv_intersection}.

\begin{subsection}{Rationalizing Matrix Coefficients} \label{rational}
Given a basis of the classical finite-dimensional simple $\uqqq$-modules, i.e. the modules defined by \eqref{i2p1} and \eqref{i2p2}, we wish to perform a re-scaling so that the coefficients are rational functions in the $q^{m_{ij}}$ variables, where the $m_{ij}$ are entries in the corresponding GT pattern. (Currently \eqref{a2pj} and \eqref{b2pj} are square roots of rational functions in the $q^{m_{ij}}$ variables.) This is necessary to prove Theorem \ref{existence_gen_n} since we will be able to use Lemma \ref{zar} \eqref{zar4}. We introduce the following convenient row-vector notation for describing a GT pattern $\alpha$:
\begin{equation*}
    \alpha=\begin{pmatrix}
    m_{n1} & \dots & m_{nk}\\
    & \vdots &\\
    & m_{31} &\\
    & m_{21} &
    \end{pmatrix}=:\begin{pmatrix}
    m_n\\
    \vdots\\
    m_3\\
    m_2
    \end{pmatrix}.
\end{equation*}Given a particular basis vector $\ket{\alpha}$, we define our new re-scaled vector by $\ket{\alpha}'=\mu_n(\alpha)\cdot\ket{\alpha}$, where $\mu_n(\alpha)$ is defined recursively by
\begin{equation} \label{lambda_def}
    \mu_n(\alpha):=\lambda_n\begin{pmatrix}
    m_n\\
    m_{n-1}
    \end{pmatrix}\cdot\mu_{n-1}\begin{pmatrix}
    m_{n-1}\\
    \vdots\\
    m_3\\
    m_2
    \end{pmatrix}=\lambda_n\begin{pmatrix}
    m_n\\
    m_{n-1}
    \end{pmatrix}\cdot \lambda_{n-1}\begin{pmatrix}
    m_{n-1}\\
    m_{n-2}
    \end{pmatrix}\dots\lambda_3\begin{pmatrix}
    m_3\\
    m_2
    \end{pmatrix}.
\end{equation}One can check that the rationalization provided in Proposition \ref{rat} coincides with the rationalization for when $n=3$ and $n=4$ given by the author in \cite[Lemmas~3.1,4.1]{disch}.

\begin{proposition} \label{rat}
Let $q=e^h$ where $h\in\R$. There exists a re-scaling of the basis of the classical finite-dimensional simple modules of $\uqqq$ which yield the following formulas:
\begin{subequations} \label{rat_formulas}
\begin{equation} 
    I_{2p+1,2p}.\ket{\alpha}=\sum_{j=1}^p a_{2p}^{j}(\alpha)\ket{\alpha_{2p}^{+j}}-\sum_{j=1}^{p}\hat{a}_{2p}^j(\alpha)\ket{{\alpha}_{2p}^{-j}}, \label{rat_formula_odd}
\end{equation}
\begin{equation}
    I_{2p,2p-1}.\ket{\alpha}=\sum_{j=1}^{p-1} b_{2p-1}^j(\alpha)\ket{\alpha_{2p-1}^{+j}}-\sum_{j=1}^{p-1}\hat{b}_{2p-1}^j(\alpha)\ket{{\alpha}_{2p-1}^{-j}}+i\cdot c_{2p-1}(\alpha)\ket{\alpha}, \label{rat_formula_even}
\end{equation}
\end{subequations}where, using the $l$-coordinate notation from \eqref{coeffsnoabs}, we have
\begin{subequations} \label{rat_coeff}
\begin{multline}\label{rat_coeff_odd_up}
    a_{2p}^j(\alpha)=\dfrac{[l_j'][l_j'+1]}{[2l_j'][2l_j'+2]}\dfrac{[l_j+l_j'][l_j-l_j'-1]\prod_{r=1}^{j-1}[l_r+l_j']\prod_{r=j+1}^{p}[l_j'+l_r][l_j'-l_r+1]}{\prod_{r< j}[l_r'+l_j'][l_r'+l_j'+1]}\\
    \cdot \dfrac{\prod_{r=1}^{j}[l_r''+l_j']\prod_{r=j+1}^{p-1}[l_j'+l_r''][l_j'-l_r''+1]}{\prod_{r>j}[l_j'+l_r'][l_j'-l_r'][l_j'+l_r'+1][l_j'-l_r'+1]},
\end{multline}
\begin{equation}\label{rat_coeff_odd_down}
    \hat{a}_{2p}^j(\alpha)=[l_j'-l_j'']\prod_{r=1}^{j-1}\frac{[l_r-l_j'][l_r''-l_j']}{[l_r'-l_j'+1][l_r'-l_j']},
\end{equation}
\begin{multline}\label{rat_coeff_even_up}
    b_{2p-1}^j(\alpha)=\dfrac{[l_j'+l_j''][l_j'-l_j'']\prod_{r=1}^{j-1}[l_r'+l_j'']\prod_{r=j+1}^{p}[l_j''+l_r'][l_j''-l_r']}{\prod_{r< j}[l_r''+l_j''][l_r''+l_j''-1]}\\
    \cdot\dfrac{\prod_{r=1}^{j}[l_r'''+l_j'']\prod_{r=j+1}^{p-1}[l_j''+l_r'''][l_j''-l_r''']}{\prod_{r>j}[l_j''+l_r''][l_j''-l_r''][l_j''+l_r''-1][l_j''-l_r''+1]}\dfrac{1}{[l_j'']^2[2l_j''+1][2l_j''-1]},
\end{multline}
\begin{equation}\label{rat_coeff_even_down}
    \hat{b}_{2p-1}^j(\alpha)=[l_j''-l_j'''-1]\prod_{r=1}^{j-1}\frac{[l_r'-l_j''+1][l_r'''-l_j''+1]}{[l_r''-l_j''+1][l_r''-l_j'']},
\end{equation}and
\begin{equation}\label{rat_coeff_even_const}
    c_{2p-1}(\alpha)=\frac{\prod_{r=1}^{p}[l_r']\prod_{r=1}^{p-1}[l_{r}''']}{\prod_{r=1}^{p-1}[l_r''][l_r''-1]}.
\end{equation}
\end{subequations}
\end{proposition}

We define our re-scaling recursively via \eqref{lambda_def} in the following way:\\
Let $e_i$ be the row-vector with $1$ for the $i$-th entry and where every other entry is zero. When $n=2p+1$, we define
\begin{equation*}
    \lambda_n\begin{pmatrix}
    m_n\\
    m_{n-1}+\sum_{i=1}^p(m_{ni}-m_{n-1,i})e_i
    \end{pmatrix}:=1
\end{equation*}and
\begin{multline} \label{rat_recursion_odd}
    \frac{\lambda_n\begin{pmatrix}
    m_n\\
    m_{n-1}+\sum_{i=1}^{j-1}(m_{ni}-m_{n-1,i})e_i
    \end{pmatrix}}{\lambda_n\begin{pmatrix}
    m_n\\
    m_{n-1}+\sum_{i=1}^{j-1}(m_{ni}-m_{n-1,i})e_i+e_j
    \end{pmatrix}}
    :=\bigg(\frac{[l_j'-p+j][l_j'+1]}{[2l_j'-2p+2j][2l_j'+2]}\cdot\\
    \cdot \frac{[l_j+l_j'][l_j-l_j'-1]\prod_{r=j+1}^p[l_j'+l_r][l_j'-l_r+1]}{\prod_{r>j}[l_j'+l_r'+1][l_j'-l_r'+1]}\bigg)^{1/2}.
\end{multline}

\noindent When $n=2p$, we define
\begin{equation*}
    \lambda_n\begin{pmatrix}
    m_n\\
    m_{n-1}+\sum_{i=1}^{p-1}(m_{ni}-m_{n-1,i})e_i
    \end{pmatrix}:=1
\end{equation*}and
\begin{multline} \label{rat_recursion_even}
    \frac{\lambda_n\begin{pmatrix}
    m_n\\
    m_{n-1}+\sum_{i=1}^{j-1}(m_{ni}-m_{n-1,i})e_i
    \end{pmatrix}}{\lambda_n\begin{pmatrix}
    m_n\\
    m_{n-1}+\sum_{i=1}^{j-1}(m_{ni}-m_{n-1,i})e_i+e_j
    \end{pmatrix}}
    :=\bigg(\frac{[l_j''-p+j]}{[2l_j''-2p+2j]}\cdot\\
    \cdot\frac{[l_j'+l_j''][l_j'-l_j'']\prod_{r=j+1}^{p}[l_j''+l_r'][l_j''-l_r']}{\prod_{r>j}[l_j''+l_r''][l_j''-l_r''+1]}\frac{1}{[l_j''][2l_j''+1]}\bigg)^{1/2}.
\end{multline}
Before proving Proposition \ref{rat}, we need the following technical lemma. The proof of this lemma is moved to the appendix (see Section \ref{appendix}).
\begin{lemma} \label{rat_telesc_prod}
\begin{enumerate}[(a)]
    \item \label{rat_telesc_prod_odd}The following is true for $n=2p+1$.
    \begin{enumerate}[{\rm (i)}]
        \item \label{rat_telesc_prod_odd_i}When $1\leq j \leq p$,
        \begin{multline*}
            \frac{\lambda_n\begin{pmatrix}
    m_n\\
    m_{n-1}
    \end{pmatrix}}{\lambda_n\begin{pmatrix}
    m_n\\
    m_{n-1}+e_j
    \end{pmatrix}}=\bigg(\prod_{s=1}^{j-1}\frac{[l_s+l_j'][l_s'-l_j']}{[l_s'+l_j'+1][l_s-l_j'-1]}\bigg)^{1/2}\\
    \cdot \bigg(\frac{[l_j'-p+j][l_j'+1]}{[2l_j'-2p+2j][2l_j'+2]}\frac{[l_j+l_j'][l_j-l_j'-1]\prod_{r=j+1}^p[l_j'+l_r][l_j'-l_r+1]}{\prod_{r>j}[l_j'+l_r'+1][l_j'-l_r'+1]}\bigg)^{1/2}.
        \end{multline*}
        \item \label{rat_telesc_prod_odd_ii}When $1\leq j<p$, 
        \begin{multline*}
            \frac{\lambda_{n-1}\begin{pmatrix}
            m_{n-1}\\
            m_{n-2}
            \end{pmatrix}}{\lambda_{n-1}\begin{pmatrix}
            m_{n-1}+e_j\\
            m_{n-2}
            \end{pmatrix}}=\bigg(\prod_{s=1}^{j-1}\frac{[l_s''+l_j'][l_s'-l_j'-1]}{[l_s'+l_j'][l_s''-l_j'-1]}\bigg)^{1/2}\bigg(\frac{[l_j'+l_j'']}{[2l_j'][l_j'-l_j''+1]}\bigg)^{1/2}\\
    \cdot \bigg(\frac{[2l_j'-2p+2j][l_j']}{[l_j'-p+j]}\frac{\prod_{r>j}[l_j'+l_r''][l_j'-l_r''+1]}{\prod_{r=j+1}^p[l_j'+l_r'][l_j'-l_r']}\bigg)^{1/2}.
        \end{multline*}
        \item \label{rat_telesc_prod_odd_iii}When $j=p$, 
        \begin{equation*}
            \frac{\lambda_{n-1}\begin{pmatrix}
            m_{n-1}\\
            m_{n-2}
            \end{pmatrix}}{\lambda_{n-1}\begin{pmatrix}
            m_{n-1}+e_j\\
            m_{n-2}
            \end{pmatrix}}=\bigg(\prod_{s=1}^{p-1}\frac{[l_s''+l_p'][l_s'-l_p'-1]}{[l_s'+l_p'][l_s''-l_p'-1]}\bigg)^{1/2}.
        \end{equation*}
    \end{enumerate}
    \item \label{rat_telesc_prod_even}The following is true for $n=2p$.
    \begin{enumerate}[{\rm (i)}]
        \item \label{rat_telesc_prod_even_i}When $1\leq j\leq p-1$,
        \begin{multline*}
            \frac{\lambda_n\begin{pmatrix}
    m_n\\
    m_{n-1}
    \end{pmatrix}}{\lambda_n\begin{pmatrix}
    m_n\\
    m_{n-1}+e_j
    \end{pmatrix}}=
    \bigg(\prod_{s=1}^{j-1}\frac{[l_s'+l_j''][l_s''-l_j'']}{[l_s''+l_j''][l_s'-l_j'']}\bigg)^{1/2}\\
    \cdot \bigg(\frac{[l_j''-p+j]}{[2l_j''-2p+2j]}\frac{[l_j'+l_j''][l_j'-l_j'']\prod_{r=j+1}^{p}[l_j''+l_r'][l_j''-l_r']}{\prod_{r>j}[l_j''+l_r''][l_j''-l_r''+1]}\frac{1}{[l_j''][2l_j''+1]}\bigg)^{1/2}.
        \end{multline*}
        \item \label{rat_telesc_prod_even_ii}When $1\leq j\leq p-1$, 
        \begin{multline*}
            \frac{\lambda_{n-1}\begin{pmatrix}
            m_{n-1}\\
            m_{n-2}
            \end{pmatrix}}{\lambda_{n-1}\begin{pmatrix}
            m_{n-1}+e_j\\
            m_{n-2}
            \end{pmatrix}}=\bigg(\prod_{s=1}^{j-1}\frac{[l_s'''+l_j''][l_s''-l_j''-1]}{[l_s''+l_j''-1][l_s'''-l_j'']}\bigg)^{1/2}\bigg(\frac{[l_j''+l_j''']}{[2l_j''-1][l_j''-l_j''']}\bigg)^{1/2}\\
    \cdot \bigg(\frac{[2l_j''-2p+2j]}{[l_j''-p+j][l_j'']}\frac{\prod_{r>j}[l_j''+l_r'''][l_j''-l_r''']}{\prod_{r=j+1}^{p-1}[l_j''+l_r''-1][l_j''-l_r'']}\bigg)^{1/2}.
        \end{multline*}
    \end{enumerate}
\end{enumerate}
\end{lemma}

\begin{proof}[Proof of Proposition~\ref{rat}]
First, we ensure that this is a proper re-scaling, i.e. $\lambda_n\begin{pmatrix}
m_n\\
m_{n-1}
\end{pmatrix}$ is nonzero and defined for any GT pattern. By \eqref{rescale_lambda_n}, it is enough to show that \eqref{rat_recursion_odd} and \eqref{rat_recursion_even} are nonzero and defined for any GT pattern. Since $\frac{[m]}{[2m]}=\frac{1}{q^m+q^{-m}}$, $\frac{[l_j'-p+j][l_j'+1]}{[2l_j'-2p+2j][2l_j'+2]}$ from \eqref{rat_recursion_odd} and $\frac{[l_j'-p+j]}{[2l_j'-2p+2j]}$ from \eqref{rat_recursion_even} are always nonzero. By Lemma \ref{lemq} \eqref{lemq3}, these are also defined since we are choosing $q$ to be a positive real number. For the other $q$-numbers $[a]$ in these expressions, it follows from Lemma \ref{lemq} \eqref{lemq1} that $a\neq 0$ implies that $[a]\neq 0$. We argue then that each $a$ is nonzero. 

It follows from the definitions of $l_j$ and $l_j'$ that in \eqref{rat_recursion_odd}, $l_j-l_j'-1=0$ if and only if $m_{n,j}-m_{n-1,j}=0$. But if this is true, then on the left-hand side of \eqref{rat_recursion_odd} in the denominator, the entries in the pattern described there fail to satisfy \eqref{interlacing_odd}. Thus the use of \eqref{rat_recursion_odd} is invalid in this instance. We obtain a similar result for $l_j'-l_j''$ in \eqref{rat_recursion_even} since the entries described in the pattern in the denominator of the left-hand side of \eqref{rat_recursion_even} fail to satisfy \eqref{interlacing_even}. Similarly, from the definitions of $l_j$, $l_j'$, $l_j''$, and $l_j'''$, and from \eqref{interlacing_conditions}, it is simple to show that every other $a$ is nonzero. Therefore every $q$-number from these equations is nonzero, and we have also shown that the numbers are nonzero when $q\to 1$.

We chose $q=e^h$ where $h\in\R$ since this ensures that $q$-numbers $[a]\in \R_{\geq 0}$ when $a\in \R_{\geq 0}$, hence the square root behaves multiplicatively in this proof. Now we check the case when $n=2p+1$. Recall that $\ket{\alpha}':=\mu_n(\alpha)\ket{\alpha}$ is a re-scaled basis vector and $\alpha_k^{\pm j}$ is the GT pattern $\alpha$ but where $m_{kj}$ is replaced by $m_{kj}\pm 1$. So
\begin{equation*}
    I_{n,n-1}.\ket{\alpha}'=\sum_{j=1}^p\frac{\mu_n(\alpha)}{\mu(\alpha_{n-1}^{+j})}A_{2p}^j(\alpha)\ket{\alpha_{n-1}^{+j}}'-\sum_{j=1}^p\frac{\mu_n(\alpha)}{\mu({\alpha}_{n-1}^{-j})}A_{2p}^j({\alpha}_{n-1}^{-j})\ket{{\alpha}_{n-1}^{-j}}'
\end{equation*}
\begin{multline*}
    =\sum_{j=1}^p\frac{\lambda_n\begin{pmatrix}
    m_n\\
    m_{n-1}
    \end{pmatrix}\lambda_{n-1}\begin{pmatrix}
    m_{n-1}\\
    m_{n-2}
    \end{pmatrix}}{\lambda_n\begin{pmatrix}
    m_n\\
    m_{n-1}+e_j
    \end{pmatrix}\lambda_{n-1}\begin{pmatrix}
    m_{n-1}+e_j\\
    m_{n-2}
    \end{pmatrix}}A_{2p}^j(\alpha)\ket{\alpha_{n-1}^{+j}}'\\
    -\sum_{j=1}^p\frac{\lambda_n\begin{pmatrix}
    m_n\\
    m_{n-1}
    \end{pmatrix}\lambda_{n-1}\begin{pmatrix}
    m_{n-1}\\
    m_{n-2}
    \end{pmatrix}}{\lambda_n\begin{pmatrix}
    m_n\\
    m_{n-1}-e_j
    \end{pmatrix}\lambda_{n-1}\begin{pmatrix}
    m_{n-1}-e_j\\
    m_{n-2}
    \end{pmatrix}}A_{2p}^j({\alpha}_{n-1}^{-j})\ket{{\alpha}_{n-1}^{-j}}'.
\end{multline*}

\noindent By Lemma \ref{rat_telesc_prod} \eqref{rat_telesc_prod_odd} the coefficient for $\ket{\alpha_{n-1}^{+j}}'$ when $j<p$ is 
\begin{multline*}
    \frac{\lambda_n\begin{pmatrix}
    m_n\\
    m_{n-1}
    \end{pmatrix}}{\lambda_n\begin{pmatrix}
    m_n\\
    m_{n-1}+e_j
    \end{pmatrix}}\frac{\lambda_{n-1}\begin{pmatrix}
    m_{n-1}\\
    m_{n-2}
    \end{pmatrix}}{\lambda_{n-1}\begin{pmatrix}
    m_{n-1}+e_j\\
    m_{n-2}
    \end{pmatrix}}A_{2p}^j(\alpha)
    =\bigg(\prod_{s=1}^{j-1}\frac{[l_s+l_j'][l_s'-l_j']}{[l_s'+l_j'+1][l_s-l_j'-1]}\bigg)^{1/2}\\
    \cdot \bigg(\frac{[l_j'-p+j][l_j'+1]}{[2l_j'-2p+2j][2l_j'+2]}\frac{[l_j+l_j'][l_j-l_j'-1]\prod_{r=j+1}^p[l_j'+l_r][l_j'-l_r+1]}{\prod_{r>j}[l_j'+l_r'+1][l_j'-l_r'+1]}\bigg)^{1/2}\\
    \cdot \bigg(\prod_{s=1}^{j-1}\frac{[l_s''+l_j'][l_s'-l_j'-1]}{[l_s'+l_j'][l_s''-l_j'-1]}\bigg)^{1/2}\bigg(\frac{[l_j'+l_j'']}{[2l_j'][l_j'-l_j''+1]}\bigg)^{1/2}\\
    \cdot \bigg(\frac{[2l_j'-2p+2j][l_j']}{[l_j'-p+j]}\frac{\prod_{r>j}[l_j'+l_r''][l_j'-l_r''+1]}{\prod_{r=j+1}^p[l_j'+l_r'][l_j'-l_r']}\bigg)^{1/2}\\
    \cdot\bigg(\dfrac{[l_j'][l_j'+1]}{[2l_j'][2l_j'+2]}\dfrac{\prod_{r=1}^j[l_r+l_j'][l_r-l_j'-1]\prod_{r=j+1}^{p}[l_j'+l_r][l_j'-l_r+1]}{\prod_{r< j}[l_r'+l_j'][l_r'-l_j'][l_r'+l_j'+1][l_r'-l_j'-1]}\\
    \cdot \dfrac{\prod_{r=1}^{j-1}[l_r''+l_j'][l_r''-l_j'-1]\prod_{r=j}^{p-1}[l_j'+l_r''][l_j'-l_r''+1]}{\prod_{r>j}[l_j'+l_r'][l_j'-l_r'][l_j'+l_r'+1][l_j'-l_r'+1]}\bigg)^{1/2}
\end{multline*}
\begin{equation*}
    =\bigg(\prod_{r=1}^{j-1}\frac{[l_r'-l_j'][l_r'-l_j'-1]}{[l_r-l_j'-1][l_r''-l_j'-1]}\bigg)\frac{A_{2p}^j(\alpha)^2}{[l_j'-l_j''+1]},
\end{equation*}

\noindent and when $j=p$, the coefficient for $\ket{\alpha_{n-1}^{+j}}'$ is
\begin{multline*}
    \frac{\lambda_n\begin{pmatrix}
    m_n\\
    m_{n-1}
    \end{pmatrix}}{\lambda_n\begin{pmatrix}
    m_n\\
    m_{n-1}+e_p
    \end{pmatrix}}\frac{\lambda_{n-1}\begin{pmatrix}
    m_{n-1}\\
    m_{n-2}
    \end{pmatrix}}{\lambda_{n-1}\begin{pmatrix}
    m_{n-1}+e_p\\
    m_{n-2}
    \end{pmatrix}}A_{2p}^p(\alpha)
    =\bigg(\prod_{s=1}^{p-1}\frac{[l_s+l_p'][l_s'-l_p']}{[l_s'+l_p'+1][l_s-l_p'-1]}\bigg)^{1/2}\\
    \cdot \bigg(\frac{[l_p'-p+p][l_p'+1]}{[2l_p'-2p+2p][2l_p'+2]}[l_p+l_p'][l_p-l_p'-1]\bigg)^{1/2}
    \cdot \bigg(\prod_{s=1}^{p-1}\frac{[l_s''+l_p'][l_s'-l_p'-1]}{[l_s'+l_p'][l_s''-l_p'-1]}\bigg)^{1/2}\\
    \bigg(\dfrac{[l_p'][l_p'+1]}{[2l_p'][2l_p'+2]}\dfrac{\prod_{r=1}^p[l_r+l_p'][l_r-l_p'-1]\prod_{r=1}^{p-1}[l_r''+l_p'][l_r''-l_p'-1]}{\prod_{r< p}[l_r'+l_p'][l_r'-l_p'][l_r'+l_p'+1][l_r'-l_p'-1]}\bigg)^{1/2}
\end{multline*}
\begin{equation*}
    =\bigg(\prod_{r=1}^{p-1}\frac{[l_r'-l_p'][l_r'-l_p'-1]}{[l_r-l_p'-1][l_r''-l_p'-1]}\bigg)A_{2p}^p(\alpha)^2.
\end{equation*}

\noindent Note that
\begin{equation*}
    \frac{\lambda_n\begin{pmatrix}
    m_n\\
    m_{n-1}
    \end{pmatrix}\lambda_{n-1}\begin{pmatrix}
    m_{n-1}\\
    m_{n-2}
    \end{pmatrix}}{\lambda_n\begin{pmatrix}
    m_n\\
    m_{n-1}-e_j
    \end{pmatrix}\lambda_{n-1}\begin{pmatrix}
    m_{n-1}-e_j\\
    m_{n-2}
    \end{pmatrix}}=\Bigg(\frac{\lambda_n\begin{pmatrix}
    m_n\\
    m_{n-1}-e_j
    \end{pmatrix}}{\lambda_n\begin{pmatrix}
    m_n\\
    m_{n-1}
    \end{pmatrix}}\Bigg)\inv \Bigg(\frac{\lambda_{n-1}\begin{pmatrix}
    m_{n-1}-e_j\\
    m_{n-2}
    \end{pmatrix}}{\lambda_{n-1}\begin{pmatrix}
    m_{n-1}\\
    m_{n-2}
    \end{pmatrix}}\Bigg)\inv, 
\end{equation*}so the coefficient for $\ket{{\alpha}_{n-1}^{-j}}'$ is the reciprocal of what gets multiplied onto $A_{2p}^j(\alpha)^2$ for the $\ket{\alpha_{n-1}^{+j}}'$ coefficient, except $l_j'$ is replaced with $l_j'-1$. Therefore
\begin{multline*}
    I_{n,n-1}.\ket{\alpha}'=\sum_{j=1}^{p}\bigg(\prod_{r=1}^{j-1}\frac{[l_r'-l_j'][l_r'-l_j'-1]}{[l_r-l_j'-1][l_r''-l_j'-1]}\bigg)\frac{A_{2p}^j(\alpha)^2}{[l_j'-l_j''+1]}\ket{\alpha_{n-1}^{+j}}'\\
    -\sum_{j=1}^{p}[l_j'-l_j'']\bigg(\prod_{r=1}^{j-1}\frac{[l_r-l_j'][l_r''-l_j']}{[l_r'-l_j'+1][l_r'-l_j']}\bigg)\ket{{\alpha}_{n-1}^{-j}}'.
\end{multline*}(Note that $l_j''$ does not exist when $j=p$, hence those $q$-numbers are removed in the $j=p$ term.)

Now we check the case when $n=2p$.
\begin{equation*}
    I_{n,n-1}.\ket{\alpha}'=\sum_{j=1}^{p-1}\frac{\mu_n(\alpha)}{\mu(\alpha_{n-1}^{+j})}B_{2p-1}^j(\alpha)\ket{\alpha_{n-1}^{+j}}'-\sum_{j=1}^{p-1}\frac{\mu_n(\alpha)}{\mu({\alpha}_{n-1}^{-j})}B_{2p-1}^j({\alpha}_{n-1}^{-j})\ket{{\alpha}_{n-1}^{-j}}'+i\cdot C_{2p-1}(\alpha)\ket{\alpha}'
\end{equation*}
\begin{equation*}
    =\sum_{j=1}^{p-1}\frac{\lambda_n\begin{pmatrix}
    m_n\\
    m_{n-1}
    \end{pmatrix}\lambda_{n-1}\begin{pmatrix}
    m_{n-1}\\
    m_{n-2}
    \end{pmatrix}}{\lambda_n\begin{pmatrix}
    m_n\\
    m_{n-1}+e_j
    \end{pmatrix}\lambda_{n-1}\begin{pmatrix}
    m_{n-1}+e_j\\
    m_{n-2}
    \end{pmatrix}}B_{2p-1}^j(\alpha)\ket{\alpha_{n-1}^{+j}}'
\end{equation*}
\begin{equation*}
    -\sum_{j=1}^{p-1}\frac{\lambda_n\begin{pmatrix}
    m_n\\
    m_{n-1}
    \end{pmatrix}\lambda_{n-1}\begin{pmatrix}
    m_{n-1}\\
    m_{n-2}
    \end{pmatrix}}{\lambda_n\begin{pmatrix}
    m_n\\
    m_{n-1}-e_j
    \end{pmatrix}\lambda_{n-1}\begin{pmatrix}
    m_{n-1}-e_j\\
    m_{n-2}
    \end{pmatrix}}B_{2p-1}^j({\alpha}_{n-1}^{-j})\ket{{\alpha}_{n-1}^{-j}}' +i \cdot C_{2p-1}(\alpha)\ket{\alpha}'.
\end{equation*}

\noindent By Lemma \ref{rat_telesc_prod} \eqref{rat_telesc_prod_even}, the coefficient for $\ket{\alpha_{n-1}^{+j}}'$ is
\begin{multline*}
\frac{\lambda_n\begin{pmatrix}
m_n\\
m_{n-1}
\end{pmatrix}}{\lambda_n\begin{pmatrix}
m_n\\
m_{n-1}+e_j
\end{pmatrix}}\frac{\lambda_{n-1}\begin{pmatrix}
m_{n-1}\\
m_{n-2}
\end{pmatrix}}{\lambda_{n-1}\begin{pmatrix}
m_{n-1}+e_j\\
m_{n-2}
\end{pmatrix}}B_{2p-1}^j(\alpha)=\bigg(\prod_{s=1}^{j-1}\frac{[l_s'+l_j''][l_s''-l_j'']}{[l_s''+l_j''][l_s'-l_j'']}\bigg)^{1/2}\\
    \cdot \bigg(\frac{[l_j''-p+j]}{[2l_j''-2p+2j]}\frac{[l_j'+l_j''][l_j'-l_j'']\prod_{r=j+1}^{p}[l_j''+l_r'][l_j''-l_r']}{\prod_{r>j}[l_j''+l_r''][l_j''-l_r''+1]}\frac{1}{[l_j''][2l_j''+1]}\bigg)^{1/2}\\
    \cdot\bigg(\prod_{s=1}^{j-1}\frac{[l_s'''+l_j''][l_s''-l_j''-1]}{[l_s''+l_j''-1][l_s'''-l_j'']}\bigg)^{1/2}\bigg(\frac{[l_j''+l_j''']}{[2l_j''-1][l_j''-l_j''']}\bigg)^{1/2}\\
    \cdot \bigg(\frac{[2l_j''-2p+2j]}{[l_j''-p+j][l_j'']}\frac{\prod_{r>j}[l_j''+l_r'''][l_j''-l_r''']}{\prod_{r=j+1}^{p-1}[l_j''+l_r''-1][l_j''-l_r'']}\bigg)^{1/2}\\
    \cdot\bigg(\dfrac{\prod_{r=1}^{j}[l_r'+l_j''][l_r'-l_j'']\prod_{r=j+1}^{p}[l_j''+l_r'][l_j''-l_r']}{\prod_{r< j}[l_r''+l_j''][l_r''-l_j''][l_r''+l_j''-1][l_r''-l_j''-1]}\\
    \cdot\dfrac{\prod_{r=1}^{j-1}[l_r'''+l_j''][l_r'''-l_j'']\prod_{r=j}^{p-1}[l_j''+l_r'''][l_j''-l_r''']}{\prod_{r>j}[l_j''+l_r''][l_j''-l_r''][l_j''+l_r''-1][l_j''-l_r''+1]}\dfrac{1}{[l_j'']^2[2l_j''+1][2l_j''-1]}\bigg)^{1/2}
\end{multline*}
\begin{equation*}
    =\bigg(\prod_{r=1}^{j-1}\frac{[l_r''-l_j''][l_r''-l_j''-1]}{[l_r'-l_j''][l_r'''-l_j'']}\bigg)\frac{B_{2p-1}^j(\alpha)^2}{[l_j''-l_j''']}.
\end{equation*}

\noindent Note that
\begin{equation*}
    \frac{\lambda_n\begin{pmatrix}
    m_n\\
    m_{n-1}
    \end{pmatrix}\lambda_{n-1}\begin{pmatrix}
    m_{n-1}\\
    m_{n-2}
    \end{pmatrix}}{\lambda_n\begin{pmatrix}
    m_n\\
    m_{n-1}-e_j
    \end{pmatrix}\lambda_{n-1}\begin{pmatrix}
    m_{n-1}-e_j\\
    m_{n-2}
    \end{pmatrix}}=\Bigg(\frac{\lambda_n\begin{pmatrix}
    m_n\\
    m_{n-1}-e_j
    \end{pmatrix}}{\lambda_n\begin{pmatrix}
    m_n\\
    m_{n-1}
    \end{pmatrix}}\Bigg)\inv \Bigg(\frac{\lambda_{n-1}\begin{pmatrix}
    m_{n-1}-e_j\\
    m_{n-2}
    \end{pmatrix}}{\lambda_{n-1}\begin{pmatrix}
    m_{n-1}\\
    m_{n-2}
    \end{pmatrix}}\Bigg)\inv, 
\end{equation*}so the coefficient for $\ket{{\alpha}_{n-1}^{-j}}'$ is the reciprocal of what gets multiplied onto $B_{2p-1}^j(\alpha)^2$ for the $\ket{\alpha_{n-1}^{+j}}'$ coefficient, except $l_j''$ is replaced with $l_j''-1$. Therefore
\begin{multline*} 
    I_{n,n-1}.\ket{\alpha}'=\sum_{j=1}^{p-1}\bigg(\prod_{r=1}^{j-1}\frac{[l_r''-l_j''][l_r''-l_j''-1]}{[l_r'-l_j''][l_r'''-l_j'']}\bigg)\frac{B_{2p-1}^j(\alpha)^2}{[l_j''-l_j''']}\ket{\alpha_{n-1}^{+j}}'\\
    -\sum_{j=1}^{p-1}[l_j''-l_j'''-1]\bigg(\prod_{r=1}^{j-1}\frac{[l_r'-l_j''+1][l_r'''-l_j''+1]}{[l_r''-l_j''+1][l_r''-l_j'']}\bigg)\ket{{\alpha}_{n-1}^{-j}}'+i \frac{\prod_{r=1}^{p}[l_r']\prod_{r=1}^{p-1}[l_{r}''']}{\prod_{r=1}^{p-1}[l_r''][l_r''-1]}\ket{\alpha}'.
\end{multline*}
\end{proof}

The following corollary is an extension of Proposition \ref{rat} and it also allows us to construct generic GT modules of $\uson$ in Corollary \ref{existence_gen_n_qto1}.
\begin{corollary}\label{rat_gen_q}
The rationalized matrix coefficients from Proposition \ref{rat} hold for any $q\neq 0$ where $q$ is not a root of unity, and also when $q\to 1$.
\end{corollary}
\begin{proof}
The $q\to 1$ case is included in Proposition \ref{rat} by choosing $h=0$, i.e. writing half-integers in the matrix coefficients instead of $q$-numbers of half-integers. Therefore we have rationalized matrix coefficients for finite-dimensional simple modules of $\uson$. Now for when $q$ is not a root of unity, let $u$ be a relation of $\uqqq$, i.e. a nonzero element of the kernel of the natural projection $F$ from the free algebra $\C\langle I_{21},I_{32},...,I_{n,n-1}\rangle$ to $\uqqq$. Then for a basis vector $\ket{\alpha}$ in the finite-dimensional simple module $V_{m_n}$, we want $F(u).\ket{\alpha}=0$. Note that the coefficients for the vector $F(u).\ket{\alpha}$ are rational functions in $q^{m_{ij}}$ for all entries $m_{ij}$ in the GT pattern $\alpha$. By Proposition \ref{rat}, these rational functions are $0$ when evaluated at $q^{m_{ij}}$ where $q=e^h$ for $h\in\R\setminus\{0\}$. Then by Lemma \ref{zar} \eqref{zar4}, the rational functions are identically $0$.
\end{proof}
\end{subsection}

\begin{subsection}{Existence of Generic Gelfand-Tsetlin Modules} \label{exist_gen_gz_mods}
We construct generic GT modules of $\uqqq$ and $\uson$. In the case of $\uqqq$ these are infinite-dimensional analogs to the finite-dimensional simple modules from Section \ref{fd_simple_mods}. The generic GT modules of $\uson$ are obtained from the matrix coefficients for $\uqqq$, replacing $q$-numbers with their limit as $q\to 1$. Because of Corollary \ref{rat_gen_q}, references to the formulas in Proposition \ref{rat} are including the more general case when $q\neq 0$ is not a root of unity. While the finite-dimensional simple modules are characterized by the highest weight $m_n$, i.e. the top row of a GT pattern, we need a complex number for each entry in a GT pattern in order to describe generic GT modules. In particular, for each entry below the top row we introduce a complex number $m_{ij}^0$ to parameterize our basis, while the entries in $m_n$ remain only because they appear in the formulas. This is because generic modules are described by \emph{admissible patterns}, seen in the following definition, in which the integral interlacing conditions \eqref{highest_wt_odd}--\eqref{interlacing_even} no longer need to be satisfied. We merely impose conditions on the entries of an admissible pattern such that the matrix coefficients are defined for all integer shifts of the entries $m_{ij}^0$. 

\begin{definition}
Let $m_{nj}\in \C$ for $1\leq j \leq \lfloor \frac{n}{2}\rfloor$ and let $m_{ij}^0\in\C$ for $2\leq i \leq n-1$ and $1\leq j\leq \lfloor\frac{i}{2}\rfloor$, where for all $k\in\Z$ we have the following: $q^{2m_{ij_1}^0+2m_{ij_2}^0+2k}\neq 1$ and $q^{2m_{ij_1}^0-2m_{ij_2}^0+2k}\neq 1$ for $j_1\neq j_2$, for even $i$ we have $q^{2m_{ij}^0+2k}\neq -1$ (see Lemma \ref{lemq} \eqref{lemq3}), and for odd $i$ we have $q^{2m_{ij}^0+2k}\neq 1$ and $q^{4m_{ij}^0+2k}\neq 1$.  Let $N:=\sum_{i=2}^{n-1}\lfloor\frac{i}{2}\rfloor$ and $m_n:=(m_{n1},m_{n2},...,m_{n,\lfloor\frac{n}{2}\rfloor})$. Choose $\alpha^0$ to be the pattern, or rather element of $\C^N$, with entries $m_{ij}^0$. We call such a pattern \emph{admissible}.
\end{definition}
\noindent Then we define
\begin{equation*}
    V_{\alpha^0}^{m_n}:=\bigoplus_{\alpha\in\alpha^0+\Z^N}\C \ket{\alpha}.
\end{equation*}We extend the definition of the so-called ``$l$-coordinates" from the classical modules given by \eqref{i2p_fd_action} to our infinite-dimensional space in the following natural way: For an entry $m_{ij}$ in the admissible pattern $\alpha\in\alpha^0+\Z^N$, $l_{ij}:=m_{ij}+\lfloor\frac{i+1}{2}\rfloor-s$. Also, as in $\eqref{i2p_fd_action}$ the pattern $\alpha_i^{\pm j}$ denotes the admissible pattern $\alpha$ but adding $\pm 1$ to the entry $m_{ij}$.
\begin{theorem}\label{existence_gen_n}
$V_{\alpha^0}^{m_n}$ is a module of $\uqqq$, and the action is defined by \eqref{rat_formulas}.
\end{theorem}
\begin{proof}
The above conditions on the $m_{ij}^0$ ensure that \eqref{rat_coeff_odd_up}--\eqref{rat_coeff_even_const} are defined for all $\alpha\in\alpha^0+\Z^N$. Now we follow a similar argument to \cite[Theorem~2]{maz}. Let $u$ be a relation of $\uqqq$, i.e. a nonzero element of the kernel of the natural projection $F$ from the free algebra $\C\langle I_{21},I_{32},...,I_{n,n-1}\rangle$ to $\uqqq$. We want to show that $F(u).\ket{\alpha}=0$ for all patterns $\alpha\in\alpha^0+\Z^N$. Note that the coefficients for the vector $F(u).\ket{\alpha}$ are rational functions in $q^{m_{nj}}$ for all $1\leq j\leq \lfloor\frac{n}{2}\rfloor$ and $q^{m_{ij}}$ for all entries $m_{ij}$ in $\alpha$. For $\textbf{a}\in \Z^N$ we denote the coefficient for $\ket{\alpha+\textbf{a}}$ in $F(u).\ket{\alpha}$ by $f_{\alpha+\textbf{a}}(q^{m_n},q^{\alpha})$.

Consider a classical finite-dimensional simple module $V_{\hat{m}_n}$ of $\uqqq$ with highest weight $\hat{m}_n$ and entries $\hat{m}_{ij}$. Denote a GT pattern with highest weight $\hat{m}_n$ by $\hat{\alpha}$, but for convenience we remove the highest weight (top row) from $\hat{\alpha}$. By Corollary \ref{rat_gen_q}, there exists a basis consisting of $\ket{\hat{\alpha}}$ such that the action of $\uqqq$ on this basis is given by \eqref{rat_formulas}, but of course replacing $\ket{\alpha}$ with $\ket{\hat{\alpha}}$ and $m_{ij}$ with $\hat{m}_{ij}$ in the definition of the $l$-coordinates. Note that $f_{\alpha+\textbf{a}}(q^{\hat{m}_n},q^{\hat{\alpha}})=0$ since $V_{\hat{m}_n}$ is a module. Then by Lemma \ref{zar} \eqref{zar4}, $f_{\alpha+\textbf{a}}$ is identically $0$.
\end{proof}

The following example provides an important family of generic GT modules. Because of Lemma \ref{zar} \eqref{zar_generic}--\eqref{zar_generic_classic}, this family of modules will allow us to prove statements pertaining to all of the generic GT modules $V_{\alpha^0}^{m_n}$.
\begin{example} \label{gen_gz_ex}
The ``highest weight" $m_n$ can consist of any complex numbers we want. Using Lemma \ref{lemq} \eqref{lemq1} and the fact $[b]=0$ if and only if $q^{2b}=1$, it is simple to check that $m_{ij}^0=\frac{1}{p_{ij}}$ for distinct odd primes $p_{ij}$ satisfy the conditions for the existence of the module $V_{\alpha^0}^{m_n}$.
\end{example}
\begin{corollary}\label{existence_gen_n_qto1}
$V_{\alpha^0}^{m_n}$ is a module of $\uson$ with formulas given as in \eqref{rat_formulas}, except $q$-numbers are replaced by their limit as $q\to 1$, i.e. $[a]\to a$.
\end{corollary}
\begin{proof}
As stated in the proof of Corollary \ref{rat_gen_q}, from Proposition \ref{rat} it follows that we have rationalized the matrix coefficients for finite-dimensional simple modules of $\uson$. Using Lemma \ref{zar} \eqref{zar4_classic} instead of Lemma \ref{zar} \eqref{zar4}, we are done by following the argument from Theorem \ref{existence_gen_n}.
\end{proof}
\end{subsection}

\begin{subsection}{Trivial Intersection of Annihilators} \label{triv_intersection}
For a complex unital associative algebra $A$, let $\textbf{Rep}_{A}^{\text{fin}}$ denote the category of finite-dimensional left $A$-modules, and let
\begin{equation*}
    \jfin:=\bigcap_{V\in\textbf{Rep}_{A}^{\text{fin}}}\operatorname{Ann}_A(V).
\end{equation*}
\begin{definition} \label{def_res_fin}
$A$ is \emph{residually finite} if $\jfin=(0)$.
\end{definition}\noindent This means that for any $a\in A\setminus\{0\}$ there exists a finite-dimensional representation $(V_a,\pi_a)$ of $A$ such that $\pi_a(a)\neq 0$. Harish-Chandra's Theorem states that $U(\mathfrak{g})$ is residually finite for any semisimple Lie algebra $\mathfrak{g}$, including $\mathfrak{g}=\mathfrak{so}_n$. We will need the following result:

\begin{proposition}[{\cite[Corollary]{res}}] \label{res_fin}
$\uqqq$ is residually finite.
\end{proposition}

Recall that we say the finite-dimensional modules from \cite[Section~3]{gav} and \eqref{i2p_fd_action} are \emph{classical} because they correspond to modules of $U(\mathfrak{so}_n)$ as $q\to 1$. The generic GT modules constructed in Theorem \ref{existence_gen_n} come from these. However, the finite-dimensional simple modules from \cite[Section~3]{nonclassical} are called \emph{nonclassical} since they have singularities when $q\to 1$. The following two results were proved in \cite{fd_classification}.
\begin{theorem}[{\cite[Theorem~9.3]{fd_classification}}]\label{classification_fd}
The simple modules of classical and nonclassical type exhaust all simple finite-dimensional modules of $\uqqq$ up to isomorphism. 
\end{theorem}
\begin{theorem}[{\cite[Corollary~9.4]{fd_classification}}]\label{fin_compl_red}
Finite-dimensional modules of $\uqqq$ are semisimple.
\end{theorem}

Thus all finite-dimensional modules of $\uqqq$ have been classified. The following result is a strengthening of Proposition \ref{res_fin}:
\begin{lemma}\label{ann_noncl}
If $u\in\uqqq$ annihilates every classical module of $\uqqq$, then $u=0$.
\end{lemma}
\begin{proof}
In \cite[Section~3]{nonclassical} the nonclassical modules are constructed as follows: the basis is a subset of the basis of a classical simple module whose entries are all half-integers (not integers) and whose highest weight consists of all positive entries. (The nonclassical basis is a proper subset of the classical basis because all entries must be positive in patterns of nonclassical type. Negative entries can still occur in classical patterns even if the highest weight consists of all positive entries.) The nonclassical formulas are obtained from the classical formulas by shifting the $m_{ij}$ by complex numbers of the form $\pm i \pi/2h$, where $h$ comes from $q=e^h$. (These modules are in fact direct sums of simple nonclassical modules.) For now, assume $h\in\R\setminus\{0\}$, as we did in Proposition \ref{rat}. Using the rationalized formulas from Proposition \ref{rat}, applying these shifts grants us rational formulas for the (semisimple) nonclassical modules. Indeed, in the proof of Proposition \ref{rat}, we see that the coefficients become rationalized through the squaring or cancellation of all the $q$-numbers present in the classical formulas. So we apply the aforementioned complex shifts to the $q$-numbers in \eqref{rat_recursion_odd} and \eqref{rat_recursion_even}, then use them to re-scale the basis vectors in the (semisimple) nonclassical modules. (The recursions in \eqref{rat_recursion_odd} and \eqref{rat_recursion_even} are defined in terms of raising the entries in the lower row by $1$, so this recursion is well-defined in the nonclassical case even though the basis is a proper subset of a certain classical basis.) This remains a proper re-scaling since for any real number $r$, $[r\pm i\pi/2h]=0$ if and only if $e^{2hr\pm i\pi}=e^{2hr}e^{\pm i\pi}=1$, which is impossible since $2hr\in\R$. Thus the formulas are now rationalized for $h\in\R\setminus\{0\}$. By the same argument as in Corollary \ref{rat_gen_q} (note that the $q\to 1$ case does not apply here), the formulas are thus rationalized for any $q\neq 0$ where $q$ is not a root of unity. 

For a basis vector $\ket{\alpha}$ in a classical module, $u.\ket{\alpha}$ is a linear combination of vectors whose coefficients are rational functions in $q^{m_{ij}}$. If $u$ annihilates every classical module, these rational functions are zero when evaluated at $q^{m_{ij}}$ for any collection of $m_{ij}$ forming a (classical) GT pattern of $\uqqq$. In the semisimple nonclassical modules described above, the coefficients of $u.\ket{\alpha}$ for a nonclassical pattern $\alpha$ (which is a particular type of classical pattern) are the same rational functions as in the classical case except they are evaluated at $q^{\pm i\pi/2h}q^{m_{ij}}$. By Lemma \ref{zar} \eqref{zar4}, $u$ annihilates these semisimple nonclassical modules. By Theorem \ref{classification_fd}, Theorem \ref{fin_compl_red}, and Proposition \ref{res_fin}, we are done.
\end{proof}

Now we use a density argument and earlier results in this section to obtain the following result for generic GT modules.
\begin{proposition} \label{pi_inj}
The intersection of annihilators over all generic GT modules of $\uqqq$, constructed in Theorem \ref{existence_gen_n}, is zero.
\end{proposition}
\begin{proof}
Suppose $u\in\uqqq$ annihilates every generic GT module. By Example \ref{gen_gz_ex}, we may construct an admissible pattern where the top row consists of any complex numbers and the other entries have the form $\frac{1}{p_{ij}}$ for distinct odd primes $p_{ij}$. Let $V_{\alpha^0}^{m_n}$ be a generic GT module of this form. Since the coefficients of $u.\ket{\alpha}$ are rational functions which are zero when evaluated at $q^{m_{ij}}$ for entries $m_{ij}$ in the GT pattern $\alpha$, it follows from Lemma \ref{zar} \eqref{zar_generic} that these formulas are identically $0$. Thus $u$ annihilates every classical finite-dimensional module. By Lemma \ref{ann_noncl}, $u=0$.
\end{proof}

\begin{corollary}\label{pi_inj_qto1}
The intersection of annihilators over all generic GT modules of $\uson$, given in Corollary \ref{existence_gen_n_qto1}, is zero.
\end{corollary}
\begin{proof}
Suppose $u\in\uson$ annihilates every generic GT module of $\uson$. Using the argument from Proposition \ref{pi_inj} but invoking Lemma \ref{zar} \eqref{zar_generic_classic} instead of Lemma \ref{zar} \eqref{zar_generic}, it follows that $u$ annihilates every finite-dimensional module of $\uson$. By Harish-Chandra's Theorem, $u=0$.
\end{proof}
\end{subsection}
\end{section}

\begin{section}{Embedding into a Skew Group Algebra} \label{emb}
We use generic GT modules from Theorem \ref{existence_gen_n} to embed $\uqqq$ into a skew group algebra in Section \ref{emb_quantum}, and likewise we embed $\uson$ into a skew group algebra in Section \ref{emb_classical}.

\begin{subsection}{Quantum case} \label{emb_quantum}
In this section, we assume the conditions on $q$ and assume we have admissible GT patterns as defined in Section \ref{exist_gen_gz_mods}, ensuring that the generic GT modules of $\uqqq$ exist. Now consider the generic GT module $V_{\alpha^0}^{m_n}$ of $\uqqq$, defined in Theorem \ref{existence_gen_n}. Define $r_k:=\lfloor\frac{k}{2}\rfloor$, and let $Z:=\{z_{ki}^{\pm 1}\mid 2\leq k\leq n, \ 1\leq i\leq r_k\}$ and $D:=\{\delta_{ki}\mid 2\leq k\leq n-1, \ 1\leq i\leq r_k\}$ be collections of linear maps on $V_{\alpha^0}^{m_n}$ defined as follows:
\begin{subequations}
\begin{align} \label{skew_group_action}
    z_{ki}^{\pm 1}\ket{\alpha}&=q^{\pm m_{ki}}\ket{\alpha},\\
    \delta_{ki}\ket{\alpha}&=\ket{\alpha_k^{+i}},\\
    \delta_{ki}\inv\ket{\alpha}&=\ket{{\alpha}_k^{-i}},
\end{align}
\end{subequations}where $m_{ki}$ denotes the $ki$-th entry of the admissible pattern $\alpha$. (Recall the definitions of $\alpha_k^{\pm i}$ given after \eqref{i2p_fd_action}.) Thus we have $z_{ki}$ variables corresponding to each entry in a GT pattern and a $\delta_{ki}$ shift operator corresponding to every entry below the top row, since these are the entries that are shifted under the action of generators of $\uqqq$ on $V_{\alpha^0}^{m_n}$. Let $\Lambda$ be the Laurent polynomial algebra generated by $Z$, $L:=\operatorname{Frac}\Lambda$, and $\m$ the free abelian group generated by $D$. One can check that we obtain an action of $\m$ on $L$ via
\begin{equation} \label{hopf_action}
    ~^{\delta_{ki}}z_{lj}^{\pm 1}=\begin{cases}
    q^{\mp 1} z_{lj}^{\pm 1}, & (k,i)=(l,j)\\
    z_{lj}^{\pm 1}, & \text{else}
    \end{cases}.
\end{equation}Since $\m$ is a group acting on $L$ by automorphisms, $L$ is a $\C\m$-module algebra, where $\C\m$ is the group algebra of $\m$. Thus we may define the \emph{skew group algebra} $L\#\m$, an example of a Hopf smash product. 
\begin{definition}
 $L\#\m$ is the (complex associative unital) algebra equal to $L\otimes_{\C}\C\m$ as a vector space with product determined by
\begin{equation*}
    (f\otimes \mu_1)\cdot (g\otimes \mu_2)=(f\cdot ~^{\mu_1}g)\otimes \mu_1\mu_2; \ \ \ \forall f,g\in L; \ \forall\mu_1,\mu_2\in\m.
\end{equation*}
\end{definition}
Note that for entry $m_{ij}$ in the admissible pattern $\alpha$ and for any $k\in\Z$,
\begin{equation}\label{rep_smash_prod_relationship}
    [m_{ij}+k]\ket{\alpha}=\frac{q^{m_{ij}+k}-q^{-m_{ij}-k}}{q-q\inv}\ket{\alpha}=\frac{q^kz_{ij}-q^{-k}z_{ij}\inv}{q-q\inv}\ket{\alpha}.
\end{equation}Thus for convenience, we give $\mathbf{a}_{2p}^j$, $\mathbf{\hat{a}}_{2p}^j$, $\mathbf{b}_{2p-1}^j$, $\mathbf{\hat{b}}_{2p-1}^j$, and $\mathbf{c}_{2p-1}$ as in \eqref{rat_coeff} but $[m_{ij}+k]$ is replaced by $\frac{q^kz_{ij}-q^{-k}z_{ij}\inv}{q-q\inv}\in \Lambda$.
\begin{theorem} \label{embedding}
There exists an injective algebra map $\varphi:\uqqq\hookrightarrow L\#\m$;
\begin{subequations}
\begin{align} \label{gen_image_smash}
    I_{2p+1,2p}&\mapsto \sum_{j=1}^p\delta_{2p,j} \mathbf{a}_{2p}^j-\sum_{j=1}^{p}\delta_{2p,j}\inv\mathbf{\hat{a}}_{2p}^j, \\
    I_{2p,2p-1}&\mapsto \sum_{j=1}^{p-1} \delta_{2p-1,j}\mathbf{b}_{2p-1}^j-\sum_{j=1}^{p-1}\delta_{2p-1,j}\inv\mathbf{\hat{b}}_{2p-1}^j+i\cdot \mathbf{c}_{2p-1}.
\end{align}
\end{subequations}
\end{theorem}
\begin{proof}
We follow the argument from \cite[Proposition~5.9]{fut_hartwig} for $U_q(\mathfrak{gl}_n)$. Let $F$ be the free complex associative unital algebra generated by $\{I_{21},I_{32},...,I_{n,n-1}\}$. Let $p:F\to\uqqq$ be the canonical projection $I_{i,i-1}\mapsto I_{i,i-1}$. Let $S$ be the multiplicative submonoid of $\Lambda\setminus\{0\}$ generated by $\{q^az_{ki}-q^{-a}z_{ki}\inv\mid a \in \Z, \; 2\leq k\leq n-1, \; 1\leq i \leq r_k\} \cup \{q^a\frac{z_{ki}}{z_{kj}}-q^{-a}(\frac{z_{ki}}{z_{kj}})\inv\mid a\in\Z, \; 2\leq k \leq n-1, \; 1\leq i<j\leq r_k\} \cup \{q^az_{ki}z_{kj}-q^{-a}(z_{ki}z_{kj})\inv\mid a\in\Z, \; 2\leq k\leq n-1, \; 1\leq i\leq j \leq r_k\}$ and let $\Lambda_S$ denote the localization of $\Lambda$ at $S$. Note that $S$ is $\m$-invariant, hence $\Lambda_S\#\m$ is a subalgebra of $L\#\m$. Let $\psi:F\to \Lambda_S\#\m$ be the algebra map given by
\begin{align*}
    I_{2p+1,2p}&\mapsto \sum_{j=1}^p\delta_{2p,j} \mathbf{a}_{2p}^j-\sum_{j=1}^{p}\delta_{2p,j}\inv\mathbf{\hat{a}}_{2p}^j, \\
    I_{2p,2p-1}&\mapsto \sum_{j=1}^{p-1} \delta_{2p-1,j}\mathbf{b}_{2p-1}^j-\sum_{j=1}^{p-1}\delta_{2p-1,j}\inv\mathbf{\hat{b}}_{2p-1}^j+i\cdot \mathbf{c}_{2p-1}.
\end{align*}Let $V$ be the direct product of all generic GT modules $V_{\alpha^0}^{m_n}$ (so an arbitrary element may have infinitely many nonzero components). Let $\pi:\uqqq\to\operatorname{End}(V)$ and $\rho:\Lambda_S\#\m\to\operatorname{End}(V)$ be the respective product representations. By Proposition \ref{pi_inj}, $\pi$ is injective. The map $\rho$ is also injective. Indeed, suppose $b\in\operatorname{Ann}_{\Lambda_S\#\m}(V)$. Write $b=\sum_{\mu\in\m}f_{\mu}\mu$ where $f_{\mu}\in \Lambda_S$, only finitely many of which are nonzero. Applying the rational function $f_{\mu}$ to a basis vector $\ket{\alpha}$ is the same as evaluating $f_{\mu}$ at $q^{m_{ij}}$ for entries $m_{ij}$ in the admissible pattern $\alpha$. Thus each $f_{\mu}$ is zero when evaluated at all the patterns given in Example \ref{gen_gz_ex}, so Lemma \ref{zar} \eqref{zar_generic} implies $f_{\mu}=0$ for all $\mu\in\m$. Therefore by the way we have chosen $\psi$ and $\rho$ we have the following commutative diagram:
\[
\begin{tikzcd}
F \arrow[r,twoheadrightarrow,"p"] \arrow[d,swap,"\psi"] & \uqqq \arrow[d,hookrightarrow,"\pi"]\\
\Lambda_S\#\m \arrow[r,hookrightarrow,"\rho"] & \operatorname{End}(V)
\end{tikzcd}.
\]
Now we argue that $\psi$ is constant on fibers of $p$. Suppose $b,c\in p\inv(a)$. Since the diagram commutes, $(\rho\circ\psi)(b-c)=0$. Since $\rho$ is injective, $\psi(b-c)=0$, hence $\psi(b)=\psi(c)$. Since $\psi$ is constant on fibers of $p$ and $\operatorname{Im}\pi\subseteq\operatorname{Im}\rho$ ($p$ is surjective and the diagram commutes), the following map is well-defined: $\varphi(a):=\psi(p\inv(a))=\rho\inv(\pi(a))$. Since $\operatorname{Im}\pi\subseteq\operatorname{Im}\rho$ and $\rho$ is injective, $\varphi$ is clearly an algebra map, and since $\pi$ is injective, it is clear that $\varphi$ is injective. Now we have the following commuting diagram:
\[
\begin{tikzcd}
F \arrow[r,twoheadrightarrow,"p"] \arrow[d,swap,"\psi"] & \uqqq \arrow[d,hookrightarrow,"\pi"] \arrow[dl,hookrightarrow,swap,"\exists! \varphi"]\\
\Lambda_S\#\m \arrow[r,hookrightarrow,"\rho"] & \operatorname{End}(V)
\end{tikzcd}.
\]
\end{proof}
\end{subsection}

\begin{subsection}{Classical case} \label{emb_classical}
By Corollary \ref{existence_gen_n_qto1}, $V_{\alpha^0}^{m_n}$ is also a module of $\uson$. We denote the coefficients from the formulas by $a_{2p}^{j(1)}(\alpha)$, $\hat{a}_{2p}^{j(1)}(\alpha)$, $b_{2p-1}^{j(1)}(\alpha)$, $\hat{b}_{2p-1}^{j(1)}(\alpha)$, and $c_{2p-1}^{(1)}(\alpha)$, as in \eqref{rat_coeff}, except we replace $q$-numbers $[a]\to a$. Define $\Lambda_1:=\C[x_{ki}\mid 2\leq k\leq n, \ 1\leq i\leq r_k]$ and $L_1:=\operatorname{Frac}\Lambda_1$. We consider the $x_{ki}$ to be linear maps on $V_{\alpha^0}^{m_n}$ defined by
\begin{equation*}
    x_{ki}\ket{\alpha}=m_{ki}\ket{\alpha}.
\end{equation*}We obtain an action of $\m$ on $L_1$ via
\begin{equation*}
    ~^{\delta_{ki}}x_{lj}=\begin{cases}
    x_{lj}-1, & (k,i)=(l,j)\\
    x_{lj}, & \text{else}
    \end{cases}.
\end{equation*}For every entry $m_{ij}$ in the admissible pattern $\alpha$ and for any $k\in\Z$,
\begin{equation*}
    (m_{ij}+k)\ket{\alpha}=(x_{ij}+k)\ket{\alpha}.
\end{equation*}Thus for convenience, we give $\mathbf{a}_{2p}^{j(1)}$, $\mathbf{\hat{a}}_{2p}^{j(1)}$, $\mathbf{b}_{2p-1}^{j(1)}$, $\mathbf{\hat{b}}_{2p-1}^{j(1)}$, and $\mathbf{c}_{2p-1}^{(1)}$ as in Corollary \ref{existence_gen_n_qto1}, but $m_{ij}+k$ is replaced by $x_{ij}+k\in\Lambda_1$.
\begin{theorem} \label{embedding_qto1}
There exists an injective algebra map $\varphi_1:\uson\hookrightarrow L_1\#\m$;
\begin{subequations}
\begin{align}
    I_{2p+1,2p}&\mapsto \sum_{j=1}^p\delta_{2p,j} \mathbf{a}_{2p}^{j(1)}-\sum_{j=1}^{p}\delta_{2p,j}\inv\mathbf{\hat{a}}_{2p}^{j(1)}, \\
    I_{2p,2p-1}&\mapsto \sum_{j=1}^{p-1} \delta_{2p-1,j}\mathbf{b}_{2p-1}^{j(1)}-\sum_{j=1}^{p-1}\delta_{2p-1,j}\inv\mathbf{\hat{b}}_{2p-1}^{j(1)}+i\cdot \mathbf{c}_{2p-1}^{(1)}.
\end{align}
\end{subequations}
\end{theorem}
\begin{proof}
Let $p_1:F\to\uson$ be the canonical projection $I_{i,i-1}\mapsto I_{i,i-1}$. Let $S$ be the multiplicative submonoid of $\Lambda_1\setminus\{0\}$ generated by $\{x_{ki}+a\mid a\in\Z, \ 2\leq k\leq n-1, \ 1\leq i\leq r_k\} \cup \{x_{ki}-x_{kj}+a\mid a\in\Z, \; 2\leq k\leq n-1, \; 1\leq i<j\leq r_k\} \cup \{x_{ki}+x_{kj}+a\mid a\in\Z, \; 2\leq k\leq n-1, \; 1\leq i\leq j\leq r_k\}$ and let $\Lambda_1^S$ denote the localization of $\Lambda_1$ at $S$. $S$ is $\m$-invariant, hence $\Lambda_1^S\#\m$ is a subalgebra of $L_1\#\m$. We obtain a commutative diagram analogous to the one in Theorem \ref{embedding}:
\[
\begin{tikzcd}
F \arrow[r,twoheadrightarrow,"p_1"] \arrow[d,swap,"\psi_1"] & \uson \arrow[d,hookrightarrow,"\pi_1"]\\
\Lambda_1^S\#\m \arrow[r,hookrightarrow,"\rho_1"] & \operatorname{End}(V)
\end{tikzcd}.
\]
The map $\pi_1$ is injective by Corollary \ref{pi_inj_qto1}. The map $\rho_1$ is injective as well. Indeed, suppose $b\in\operatorname{Ann}_{\Lambda_1^S\#\m}(V)$. Write $b=\sum_{\mu\in\m}f_{\mu}\mu$ where $f_{\mu}\in\Lambda_1^S$, only finitely many of which are nonzero. Applying the rational function $f_{\mu}$ to a basis vector $\ket{\alpha}$ is the same as evaluating $f_{\mu}$ at $m_{ij}$ for entries $m_{ij}$ in the admissible pattern $\alpha$. By Example \ref{gen_gz_ex} and Lemma \ref{zar} \eqref{zar_generic_classic}, $f_{\mu}=0$. Therefore $b=0$. Following the rest of the argument from Theorem \ref{embedding}, we have the following commutative diagram:
\[
\begin{tikzcd}
F \arrow[r,twoheadrightarrow,"p_1"] \arrow[d,swap,"\psi_1"] & \uson \arrow[d,hookrightarrow,"\pi_1"] \arrow[dl,hookrightarrow,swap,"\exists! \varphi_1"]\\
\Lambda_1^S\#\m \arrow[r,hookrightarrow,"\rho_1"] & \operatorname{End}(V)
\end{tikzcd}.
\]
\end{proof}
\end{subsection}
\end{section}

\begin{section}{Harish-Chandra Subalgebra} \label{hc}
In Section \ref{hc_quantum}, we prove that the image of the GT subalgebra $\Gamma$ of $\uqqq$ under the embedding from Theorem \ref{embedding} is an invariant algebra of a Laurent polynomial ring under the group action of a direct product of ($q$-)Weyl groups. We then show that $\Gamma$ is a Harish-Chandra subalgebra of $\uqqq$. In Section \ref{hc_classical}, we show that the image of the corresponding subalgebra $\Gamma_1$ of $\uson$ under the embedding from Theorem \ref{embedding_qto1} is an invariant algebra of a polynomial ring under the group action of the direct product of Weyl groups. Lastly, we show that $\Gamma_1$ is a Harish-Chandra subalgebra of $\uson$.

\begin{subsection}{Quantum case} \label{hc_quantum}
For $1\leq d\leq k:=\lfloor{\frac{n}{2}}\rfloor$ there are Casimir elements of $\uqqq$ denoted by $C_n^{(2d)}$ given in \cite[Theorem~1]{cas}, where $C_n^{(n)}$ is replaced by the simpler $C_n^{(n)+}$ when $n$ is even. (Since $\pi$ in the proof of Theorem \ref{embedding} is injective, we see by the eigenvalues given in the following Theorem \ref{cas_eigvals} that $(C_n^{(n)+})^2=C_n^{(n)}$. Also $C_n^{(n)+}=C_n^{(n)-}$ from \cite{cas}.) Though these elements are central in $\uqqq$, it is only conjectured that they generate the whole center \cite{cas}.

Let $\textbf{a}=(a_1,a_2,...,a_k)$ be a sequence of complex numbers. In the following theorem, we use the so-called \emph{generalized factorial elementary symmetric polynomials} \cite[Proposition~2.2]{capelli_id}:
\begin{equation*}
    e_d(y_1,...,y_k\mid \textbf{a})=\sum_{1\leq p_1<p_2<\dots<p_d\leq k}\prod_{r=1}^d (y_{p_r}-a_{p_r-r+1}).
\end{equation*}
\begin{theorem}[{\cite[Theorem~2]{cas}}] \label{cas_eigvals}
The eigenvalue of the operator $C_n^{(2d)}$ on classical finite-dimensional simple $\uqqq$-modules is
\begin{equation*}
    \chi_{m_n}^{(2d)}=(-1)^de_d([s_{n1}]^2,[s_{n2}]^2,...,[s_{nk}]^2\mid \textbf{a})
\end{equation*}where $\textbf{a}=([\epsilon]^2,[\epsilon+1]^2,[\epsilon+2]^2,...)$, $s_{ni}=m_{ni}+k-i+\epsilon$. Here $\epsilon=0$ when $n=2k$ and $\epsilon=\frac{1}{2}$ when $n=2k+1$.

When $n=2k$,
\begin{equation*}
    \pi_{m_n}(C_n^{(n)+})=(\sqrt{-1})^k\bigg(\prod_{r=1}^k [s_{nr}]\bigg)\mathbf{1}.
\end{equation*}
\end{theorem}

\begin{definition}
Let $\Gamma$ be the (commutative) subalgebra of $\uqqq$ generated by $C_i^{(2d)}$ for $2\leq i\leq n$ and $1\leq d\leq k$, except we replace $C_i^{(i)}$ with $C_i^{(i)+}$ when $i$ is even. We define $C_2^{(2)+}=I_{21}$. We call $\Gamma$ the \emph{Gelfand-Tsetlin subalgebra} (GT subalgebra) of $\uqqq$. Let $\Gamma^{(n)}$ be the subalgebra of $\Gamma$ generated by $C_n^{(2d)}$ for $1\leq d\leq k$, except we replace $C_n^{(n)}$ with $C_n^{(n)+}$ when $n=2k$. Let $\Lambda^{(n)}$ be the subalgebra of $\Lambda$ (hence of $L\#\m$) generated by $z_{ni}^{\pm 1}$ for $1\leq i\leq k$.
\end{definition}

We argue in the proof of the following corollary that the generators of $\Gamma^{(n)}$ act as scalar operators on generic GT modules $V_{\alpha^0}^{m_n}$, with analogous eigenvalues to those seen in Theorem \ref{cas_eigvals}. 

\begin{corollary} \label{gamma_contained_lambda}
The Casimir elements $C_n^{(2d)}$ (and $C_n^{(n)+}$ when $n$ is even) act as scalar operators on generic GT modules $V_{\alpha^0}^{m_n}$, even when the modules are not simple. The eigenvalues of the generators are those given in Theorem \ref{cas_eigvals}, except every $m_{ni}$ comes from the $k$-tuple $m_n$ that parameterizes $V_{\alpha^0}^{m_n}$ instead of a highest weight of a finite-dimensional simple module. Moreover, $\varphi(\Gamma^{(r)})\subseteq \Lambda^{(r)}$ for $2\leq r\leq n$, thus $\varphi(\Gamma)\subseteq \Lambda$, where $\varphi:\uqqq\hookrightarrow L\#\m$ is the map from Theorem \ref{embedding}. 
\end{corollary}
\begin{proof}
Consider the generator $C_n^{(2d)}$ and basis vector $\ket{\alpha}$ in $V_{\alpha^0}^{m_n}$. (We exclude $C_n^{(n)+}$ when $n$ is even from this proof since the argument is identical.) Recall that $\alpha^0\in \C^N$ for $N:=\sum_{i=2}^{n-1}\lfloor \frac{i}{2}\rfloor$. 
\begin{equation*}
    C_n^{(2d)}.\ket{\alpha}=\sum_{\textbf{a}\in S}f_{\alpha+\textbf{a}}(q^{m_{ij}})\ket{\alpha+\textbf{a}}
\end{equation*}where $S$ is a finite subset of $\Z^N$ such that $(0,...,0)\in S$ and $f_{\alpha+\textbf{a}}(q^{m_{ij}})$ denotes a rational function evaluated at $q^{m_{ij}}$ for all entries $m_{ij}$ in the pattern $\alpha$ including the top row $m_n$. Note that since the formulas for the generators of $\uqqq$ on a basis vector in $V_{\alpha^0}^{m_n}$ are constructed to be identical to the rationalized formulas for classical finite-dimensional simple modules given in Corollary \ref{rat_gen_q}, we have the same underlying rational functions $f_{\alpha+\textbf{a}}=f_{\hat{\alpha}+\textbf{a}}$ describing the coefficients when $C_n^{(2d)}$ acts on a basis vector $\ket{\hat{\alpha}}$ in the classical modules. Then by Theorem \ref{cas_eigvals}, for all $\textbf{a}\neq (0,...,0)$, $f_{\alpha+\textbf{a}}$ is zero when evaluated at any collection of $q^{\hat{m}_{ij}}$ for entries $\hat{m}_{ij}$ defining a classical GT pattern. Therefore by Lemma \ref{zar} \eqref{zar4}, these rational functions are identically zero, so $C_n^{(2d)}$ acts diagonally on $V_{\alpha^0}^{m_n}$ with respect to the basis given in the definition of $V_{\alpha^0}^{m_n}$. Also, since $f_{\alpha}=f_{\hat{\alpha}}$, we see by Theorem \ref{cas_eigvals} that $f_{\alpha}(q^{m_{ij}})$ depends only on the entries $m_{nj}$ which characterize $V_{\alpha^0}^{m_n}$ and thus does not depend on the basis vector $\ket{\alpha}$. Thus $C_n^{(2d)}$ acts as a scalar operator on $V_{\alpha^0}^{m_n}$. It also follows from Theorem \ref{cas_eigvals} that this eigenvalue is a Laurent polynomial in the $q^{m_{ij}}$ variables.
\end{proof}

Let $z_{nj}':=q^{k-j+\epsilon}\cdot z_{nj}$. (Recall that $k=\lfloor\frac{n}{2}\rfloor$, $1\leq j \leq k$, $\epsilon=0$ when $n=2k$ and $\epsilon=\frac{1}{2}$ when $n=2k+1$.) Note that $\Lambda^{(n)}=\C[(z_{nj}')^{\pm 1}\mid 1\leq j\leq k]$. Consider the automorphisms $\sigma_{nj}$ and $\tau_{nj}$ of $\Lambda^{(n)}$ defined by
    \begin{align*}
        \sigma_{nj}((z_{nj}')^{\pm 1})&=(z_{nj}')^{\mp 1},\\
        \tau_{nj}((z_{nj}')^{\pm 1})&=-(z_{nj}')^{\pm 1},
    \end{align*}and every $(z_{nr}')^{\pm 1}$ generator is fixed for $r\neq j$. Then the group $(\Z/2\Z\times \Z/2\Z)^k$ acts on $\Lambda^{(n)}$ by automorphisms via
    \begin{equation*}
        ~^{(a_j,b_j)_{j=1}^k}f=\bigg(\prod_{j=1}^k\sigma_{nj}^{a_j}\tau_{nj}^{b_j}\bigg)(f),
    \end{equation*}
    where we employ the convention that $a_j,b_j\in\{0,1\}$ for all $1\leq j \leq k$. The symmetric group $S_k$ acts on $\Lambda^{(n)}$ by permuting the indices of the $(z_{nj}')^{\pm 1}$ variables. We also have an action of $S_k$ on $(\Z/2\Z\times \Z/2\Z)^k$ by permuting the indices of the $(a_j,b_j)$. Let $G$ be the group $W_2^q\times W_3^q\times \dots \times W_n^q$ where
    \begin{equation*}
        W_i^q:=\begin{cases}
        (\Z/2\Z\times \Z/2\Z)^k\rtimes S_k, & i=2k+1\\
        (\Z/2\Z\times \Z/2\Z)^{k}_{\text{even}}\rtimes S_{k}, & i=2k
        \end{cases}
    \end{equation*}and $(\Z/2\Z\times \Z/2\Z)^{k}_{\text{even}}$ is the kernel of the group homomorphism $h:(\Z/2\Z\times \Z/2\Z)^{k}\to\Z/2\Z$;
    \begin{equation*}
        (a_j,b_j)_{j=1}^k\mapsto\sum_{j=1}^k(a_j+b_j).
    \end{equation*}Thus $W_i^q$ acts on $\Lambda^{(i)}$ by automorphisms for each $2\leq i\leq n$, which enables us to extend to an action of $G$ on $\Lambda$ by automorphisms. The following theorem describes the image of the GT subalgebra of $\uqqq$ under the injective algebra map $\varphi:\uqqq\hookrightarrow L\#\m$ from Theorem \ref{embedding}:
\begin{theorem} \label{lambda_fg}
    $\varphi(\Gamma)=\Lambda^G$. In particular, $\Lambda$ is finitely generated as a $\varphi(\Gamma)$-module.
\end{theorem}
\begin{proof}
By the definition $\varphi(a)=\rho\inv(\pi(a))$ in the proof of Theorem \ref{embedding}, and by Corollary \ref{gamma_contained_lambda}, we see
\begin{equation*}
    \varphi(C_n^{(2d)})=(-1)^d\sum_{1\leq p_1<p_2<\dots<p_d\leq k} \; \prod_{r=1}^d \bigg(\bigg(\frac{z_{n,p_r}'-(z_{n,p_r}')\inv}{q-q\inv}\bigg)^2-a_{p_r-r+1}\bigg).
\end{equation*}Setting $b_{nj}:=(z_{nj}')^2+(z_{nj}')^{-2}$, we get
\begin{equation*}
    \varphi(C_n^{(2d)})=(-1)^d\sum_{1\leq p_1<p_2<\dots<p_d\leq k} \; \prod_{r=1}^d \bigg(\frac{b_{n,p_r}-2}{(q-q\inv)^2}-a_{p_r-r+1}\bigg).
\end{equation*}It is clear that each $b_{nj}$ is fixed under $\sigma_{nr}$ and $\tau_{nr}$ for any $r$. Also, $~^{\omega}b_{nj}=b_{n,\omega(j)}$ for all $\omega\in S_k$. It follows from \cite[Proposition~2.2]{capelli_id} that $\varphi(C_n^{(2d)})$ is symmetric in the $b_{nj}$ variables. (Note that their use of the letter $z$ differs from ours, but we obtain this result nonetheless.) When $n=2k$, $C_n^{(n)}$ is replaced by $C_n^{(n)+}$ as a generator. We have
\begin{equation*}
    \varphi(C_n^{(n)+})=(\sqrt{-1})^k\prod_{r=1}^k \frac{z_{nr}'-(z_{nr}')\inv}{q-q\inv}. 
\end{equation*}Applying $\sigma_{nr}$ or $\tau_{nr}$ negates this expression for any $r$. By the definition of $W_n^q$ when $n$ is even, only an even number of sign changes can occur under the action of $W_n^q$, and of course this expression is symmetric in the $(z_{nr}')^{\pm 1}$. Therefore $\varphi(\Gamma)\subseteq\Lambda^G$.

Before proving the reverse containment, we want to show that the elementary symmetric polynomial $e_d(b_{n1},b_{n2},...,b_{nk})\in\varphi(\Gamma^{(n)})$ for $1\leq d\leq k$. We prove this claim by strong induction on $d$. Obviously $e_1(b_{n1},b_{n2},...,b_{nk})\in \varphi(\Gamma^{(n)})$ by $\varphi(C_n^{(2)})$. Now suppose $e_i(b_{n1},b_{n2},...,b_{nk})\in\varphi(\Gamma^{(n)})$ for all $i<d$. Multiply $\varphi(C_n^{(2d)})$ by $(-1)^d(q-q\inv)^{2d}$ for convenience, thus we define
\begin{equation*}
    a:=(-1)^d(q-q\inv)^{2d}\cdot \varphi(C_n^{(2d)})=\sum_{1\leq p_1<p_2<\dots<p_d\leq k} \; \prod_{r=1}^d (b_{n,p_r}-2-(q-q\inv)^2a_{p_r-r+1})\in\varphi(\Gamma^{(n)}).
\end{equation*}Since $\varphi(C_n^{(2d)})$ is symmetric in the $b_{nj}$ variables, it follows that $a$ is symmetric in the $b_{nj}$ variables. Note that the sum of all degree $d$ terms in $a$ is $e_d(b_{n1},b_{n2},...,b_{nk})$. The sum of terms of smaller degree $a-e_d(b_{n1},b_{n2},...,b_{nk})$ is symmetric since invariant elements under the standard action of the symmetric group $S_k$ form a ring. By Newton's Theorem, $a-e_d(b_{n1},b_{n2},...,b_{nk})$ is generated by the elementary symmetric polynomials $e_i(b_{n1},b_{n2},...,b_{nk})$ for $i<d$. Then by the induction hypothesis, $a-e_d(b_{n1},b_{n2},...,b_{nk})\in\varphi(\Gamma^{(n)})$. Therefore
\begin{equation*}
    e_d(b_{n1},b_{n2},...,b_{nk})=a-(a-e_d(b_{n1},b_{n2},...,b_{nk}))\in\varphi(\Gamma^{(n)}).
\end{equation*}

Suppose $f\in(\Lambda^{(n)})^{W_n^q}$ is non-constant and $n=2k+1$. Since $f$ is fixed under the actions of all $\tau_{nj}$ for $1\leq j\leq k$, it follows that $f\in\C[(z_{nj}')^{\pm 2}\mid 1\leq j\leq k]$. Consider a monomial $(z_{n1}')^{c_1}(z_{n2}')^{c_2}\dots (z_{nk}')^{c_k}$ in $f$ where $c_j$ is an even integer for all $1\leq j\leq k$. Since $f$ is fixed under the action of every $\sigma_{nj}$, the new monomial obtained by changing the sign of any of the $c_j$ (including multiple sign changes at once) must have the same coefficient in $f$ as the original monomial. It follows that $f$ can be written as a polynomial in the $b_{nj}$ variables. Since $f$ is fixed under the action of the symmetric group $S_k$, it follows that $f$ is symmetric with respect to the $b_{nj}$ variables. By Newton's Theorem, $f$ is generated by elementary symmetric polynomials $e_d(b_{n1},b_{n2},...,b_{nk})$, which are in $\varphi(\Gamma^{(n)})$.

Now suppose $f\in(\Lambda^{(n)})^{W_n^q}$ is non-constant and $n=2k$. $f$ is fixed under the action of every $\tau_{nj_1}\tau_{nj_2}$ for $1\leq j_1,j_2\leq k$. Thus for any monomial $(z_{n1}')^{c_1}(z_{n2}')^{c_2}\dots (z_{nk}')^{c_k}$ in $f$, the $c_j$ must be all even integers or all odd integers since $c_{j_1}+c_{j_2}\in 2\Z$. Split $f=f_{\text{even}}+f_{\text{odd}}$ where $f_{\text{even}}$ consists of the monomials where the exponents are even and $f_{\text{odd}}$ consists of the monomials where the exponents are odd. Note that every monomial in $f_{\text{even}}$ is fixed under every $\tau_{nj}$ and $f$ is fixed under every $\sigma_{nj_1}\tau_{nj_2}$. Since applying any $\sigma_{nj}$ to $f_{\text{even}}$ yields a Laurent polynomial of even degree, $f_{\text{even}}$ must be fixed under every $\sigma_{nj_1}\tau_{nj_2}$, hence fixed under every $\sigma_{nj}$. Therefore, as in the previous paragraph, $f_{\text{even}}$ is generated by elementary symmetric polynomials $e_d(b_{n1},b_{n2},...,b_{nk})$, hence it is in $\varphi(\Gamma^{(n)})$. Now consider $f_{\text{odd}}$. Since $f$ is fixed under the action of every $\sigma_{nj}\tau_{nj}$ and applying any $\sigma_{nj}$ or $\tau_{nj}$ to $f_{\text{odd}}$ yields a Laurent polynomial of odd degree, $f_{\text{odd}}$ is fixed under every $\sigma_{nj}\tau_{nj}$. Then
\begin{align*}
    f_{\text{odd}}(z_{n1}',...,z_{nk}')&=\frac{1}{2}f_{\text{odd}}(z_{n1}',...,z_{nj}',...,z_{nk}')+\frac{1}{2}f_{\text{odd}}(z_{n1}',...,-(z_{nj}')\inv,...,z_{nk}')\\
    &=\frac{1}{2}f_{\text{odd}}(z_{n1}',...,z_{nj}',...,z_{nk}')-\frac{1}{2}f_{\text{odd}}(z_{n1}',...,(z_{nj}')\inv,...,z_{nk}'),
\end{align*}so $f_{\text{odd}}$ is divisible by $(z_{nj}'-1)(z_{nj}'+1)$, hence by $(z_{nj}')\inv(z_{nj}'-1)(z_{nj}'+1)=z_{nj}'-(z_{nj}')\inv$. This holds for all $1\leq j\leq k$, so $f_{\text{odd}}$ is divisible by $g:=\prod_{j=1}^k(z_{nj}'-(z_{nj}')\inv)$. In particular, $f_{\text{odd}}=g\cdot g_{\text{even}}$ where $g_{\text{even}}$ is a Laurent polynomial such that the degree of every variable in each monomial is even. Note that $g$ is in $\varphi(\Gamma^{(n)})$ by $\varphi(C_n^{(n)+})$. For all $h\in W_n^q$, 
\begin{equation*}
    g\cdot g_{\text{even}}=f_{\text{odd}}=~^hf_{\text{odd}}=~^hg\cdot ~^h g_{\text{even}}=g\cdot ~^h g_{\text{even}},
\end{equation*}so $g_{\text{even}}$ is fixed by $W_n^q$ since $\Lambda^{(n)}$ is an integral domain. Following the argument that showed that $f_{\text{even}}\in \varphi(\Gamma^{(n)})$, we see that $g_{\text{even}}\in\varphi(\Gamma^{(n)})$.
\end{proof}

In \cite{dfo} the notion of a Harish-Chandra subalgebra was defined (in a more general context). This notion has played an important role in the study of GT modules over $U(\mathfrak{gl}_n)$ and related algebras.
\begin{definition}[{\cite{dfo}}] \label{hc_subalg_def}
A commutative subalgebra $C\subseteq U$ is a \emph{Harish-Chandra subalgebra} if $C u C$ is finitely generated as a left and right $C$-module for all $u\in U$.
\end{definition}

The following lemma is well-known but we provide a proof for convenience.
\begin{lemma} \label{hc_subalg_smash_prod}
$\Lambda$ is a Harish-Chandra subalgebra of $L\#\m$.
\end{lemma}
\begin{proof}
$L\cup\m$ is a set of generators of $L\#\m$ as a ring. By \cite[Proposition~8]{dfo}, it is enough to check $\Lambda x \Lambda$ is finitely generated as a left and right $\Lambda$-module for $x\in L\cup \m$. Since $L$ is commutative, it is enough to check for $\mu\in\m$. Clearly $\Lambda\mu\Lambda\subseteq\Lambda\mu$ and $\Lambda\mu\Lambda\subseteq\mu\Lambda$ by the action of $\m$ on $\Lambda$. The reverse inclusions are obvious.
\end{proof}

We are now ready to prove the second main theorem in this section.
\begin{theorem} \label{hc_subalg}
    $\varphi(\Gamma)$ is a Harish-Chandra subalgebra of $L\#\m$. Furthermore, $\Gamma$ is a Harish-Chandra subalgebra of $\uqqq$.
\end{theorem}
\begin{proof}
Let $\Gamma':=\varphi(\Gamma)$, and let $x\in L\#\m$. By Lemma \ref{hc_subalg_smash_prod}, there exist $x_1,x_2,...,x_r\in \Lambda x\Lambda$ such that $\Lambda x\Lambda=\Lambda x_1+\Lambda x_2+\dots+\Lambda x_r$. By Theorem \ref{lambda_fg}, there exist $f_1,f_2,...,f_s\in\Lambda$ such that $\Lambda=\Gamma'f_1+\Gamma'f_2+\dots+\Gamma'f_s$. Then
\begin{equation*}
    \Gamma'x\Gamma'\subseteq\Lambda x\Lambda \subseteq\Lambda x_1+\Lambda x_2+\dots+\Lambda x_r \subseteq\sum_{i=1}^s\sum_{j=1}^r \Gamma'f_i x_j.
\end{equation*}Since $\Gamma'$ is a finitely generated $\C$-algebra, it is a Noetherian ring. Therefore $\sum_{i=1}^s\sum_{j=1}^r \Gamma'f_i x_j$ is a Noetherian left $\Gamma'$-module, hence $\Gamma' x \Gamma'$ is finitely generated as a left $\Gamma'$-module. A similar argument shows that $\Gamma'x\Gamma'$ is finitely generated as a right $\Gamma'$-module. Since $\varphi$ is injective, $\Gamma\cong\Gamma'$ and $\uqqq$ is isomorphic to a subalgebra of $L\#\m$. Therefore $\Gamma$ is a Harish-Chandra subalgebra of $\uqqq$.
\end{proof}
\end{subsection}

\begin{subsection}{Classical case} \label{hc_classical}
We emphasize again that the associative algebra $\uqqq$ is defined for all $q\in\C\setminus\{0\}$, including $q=1$ when it becomes isomorphic to $\uson$. The Casimir elements of $\uqqq$ given in \cite[Theorem~1]{cas} are defined for all $q\in\C\setminus\{0\}$, including $q\to 1$, so in this subsection we denote the Casimir elements of $\uson$ by $C_n^{(2d)}$ and $C_n^{(n)+}$ as in the previous subsection. These elements act as scalar operators on finite-dimensional simple modules of $\uson$ with eigenvalues given in Theorem \ref{cas_eigvals}, except $q$-numbers $[a]$ are replaced by their limit $a$ as $q\to 1$. We see in Corollary \ref{existence_gen_n_qto1} that generic GT modules $V_{\alpha^0}^{m_n}$ are modules of $\uson$. We will argue in the proof of the following corollary of Theorem \ref{cas_eigvals} that the elements of $\uson$ act as scalar operators on modules $V_{\alpha^0}^{m_n}$, with analogous eigenvalues to those given in finite-dimensional simple modules. We begin with a definition.

\begin{definition}
Let $\Gamma_1$ be the subalgebra generated by the elements $C_i^{(2d)}$ for $2\leq i\leq n$ and $1\leq d \leq k:=\lfloor\frac{n}{2}\rfloor$, except we replace $C_i^{(i)}$ with $C_i^{(i)+}$ when $i$ is even. We define $C_2^{(2)+}=I_{21}$. This is the \emph{Gelfand-Tsetlin subalgebra} (GT subalgebra) of $\uson$. Let $\Gamma_1^{(n)}$ be the subalgebra of $\Gamma_1$ generated by $C_n^{(2d)}$ for $1\leq d\leq k$, except we replace $C_n^{(n)}$ with $C_n^{(n)+}$ when $n=2k$. Let $\Lambda_1^{(n)}$ be the subalgebra of $\Lambda_1$ (hence of $L_1\#\m$) generated by $x_{ni}$ for $1\leq i\leq k$.
\end{definition}

\begin{corollary} \label{gamma_contained_lambda_qto1}
The Casimir elements $C_n^{(2d)}$ (and $C_n^{(n)+}$ when $n$ is even) act as scalar operators on generic GT modules $V_{\alpha^0}^{m_n}$ of $\uson$, even when the modules are not simple. The eigenvalues of the generators are those given in Corollary \ref{gamma_contained_lambda}, except the $q$-numbers $[a]$ are replaced by their limit $a$ as $q\to 1$. Moreover, $\varphi_1(\Gamma_1^{(r)})\subseteq\Lambda_1^{(r)}$ for $2\leq r\leq n$, thus $\varphi_1(\Gamma_1)\subseteq\Lambda_1$, where $\varphi_1:\uson\hookrightarrow L_1\#\m$ is the injective algebra map from Theorem \ref{embedding_qto1}.
\end{corollary}
\begin{proof}
Consider the generator $C_n^{(2d)}$ and basis vector $\ket{\alpha}$ in $V_{\alpha^0}^{m_n}$. (We exclude $C_n^{(n)+}$ when $n$ is even from this proof since the argument is identical.) Recall that $\alpha^0\in \C^N$ for $N:=\sum_{i=2}^{n-1}\lfloor \frac{i}{2}\rfloor$. 
\begin{equation*}
    C_n^{(2d)}.\ket{\alpha}=\sum_{\textbf{a}\in S}f_{\alpha+\textbf{a}}({m_{ij}})\ket{\alpha+\textbf{a}}
\end{equation*}where $S$ is a finite subset of $\Z^N$ such that $(0,...,0)\in S$ and $f_{\alpha+\textbf{a}}({m_{ij}})$ denotes a rational function evaluated at entries $m_{ij}$ in the pattern $\alpha$ including the top row $m_n$. Note that since the formulas for the generators of $\uson$ on a basis vector in $V_{\alpha^0}^{m_n}$ are constructed to be identical to the rationalized formulas for finite-dimensional simple modules given in Corollary \ref{rat_gen_q}, we have the same underlying rational functions $f_{\alpha+\textbf{a}}=f_{\hat{\alpha}+\textbf{a}}$ describing the coefficients when $C_n^{(2d)}$ acts on a basis vector $\ket{\hat{\alpha}}$ in the finite-dimensional modules. Then by Theorem \ref{cas_eigvals} for $q\to 1$, for all $\textbf{a}\neq (0,...,0)$, $f_{\alpha+\textbf{a}}$ is zero when evaluated at any collection of entries $\hat{m}_{ij}$ defining a classical GT pattern. Therefore by Lemma \ref{zar} \eqref{zar4_classic}, these rational functions are identically zero, so $C_n^{(2d)}$ acts diagonally on $V_{\alpha^0}^{m_n}$ with respect to the basis given in the definition of $V_{\alpha^0}^{m_n}$. Also, since $f_{\alpha}=f_{\hat{\alpha}}$, we see by Theorem \ref{cas_eigvals} that $f_{\alpha}({m_{ij}})$ depends only on the entries $m_{nj}$ which characterize $V_{\alpha^0}^{m_n}$ and thus does not depend on the basis vector $\ket{\alpha}$. Thus $C_n^{(2d)}$ acts as a scalar operator on $V_{\alpha^0}^{m_n}$. It also follows from Theorem \ref{cas_eigvals} when $q\to 1$ that this eigenvalue is a polynomial in the ${m_{ij}}$ variables.
\end{proof}

Let $x_{nj}':=x_{nj}+k-j+\epsilon$. (Recall that $k=\lfloor\frac{n}{2}\rfloor$, $1\leq j \leq k$, $\epsilon=0$ when $n=2k$ and $\epsilon=\frac{1}{2}$ when $n=2k+1$.) Note that $\Lambda_1^{(n)}=\C[x_{nj}'\mid 1\leq j\leq k]$. Consider the automorphism $\sigma_{nj}$ of $\Lambda_1^{(n)}$ defined by
    \begin{equation*}
        \sigma_{nj}(x_{nj}')=-x_{nj}',
    \end{equation*}and every $x_{nr}'$ generator is fixed for $r\neq j$. Then the group $(\Z/2\Z)^k$ acts on $\Lambda_1^{(n)}$ by automorphisms via
    \begin{equation*}
        ~^{(a_j)_{j=1}^k}f=\bigg(\prod_{j=1}^k\sigma_{nj}^{a_j}\bigg)(f),
    \end{equation*}
    where we employ the convention that $a_j\in\{0,1\}$ for all $1\leq j \leq k$. The symmetric group $S_k$ acts on $\Lambda_1^{(n)}$ by permuting the indices of the $x_{nj}'$ variables. We also have an action of $S_k$ on $(\Z/2\Z)^k$ by permuting the indices of the $a_j$. Let $G$ be the group $W_2\times W_3\times \dots \times W_n$ where each $W_i$ is the Weyl group of the Lie algebra $\mathfrak{so}_i$. Explicitly,
    \begin{equation*}
        W_i:=\begin{cases}
        (\Z/2\Z)^k\rtimes S_k, & i=2k+1\\
        (\Z/2\Z)^{k}_{\text{even}}\rtimes S_{k}, & i=2k
        \end{cases}
    \end{equation*}and $(\Z/2\Z)^{k}_{\text{even}}$ is the kernel of the group homomorphism $h:(\Z/2\Z)^{k}\to\Z/2\Z$;
    \begin{equation*}
        (a_j)_{j=1}^k\mapsto\sum_{j=1}^k a_j.
    \end{equation*}Thus $W_i$ acts on $\Lambda_1^{(i)}$ by automorphisms for each $2\leq i\leq n$, which enables us to extend to an action of $G$ on $\Lambda_1$ by automorphisms. The following theorem describes the image of the GT subalgebra of $\uson$ under the injective algebra map $\varphi_1:\uson\hookrightarrow L_1\#\m$ from Theorem \ref{embedding_qto1}.
    
\begin{theorem} \label{lambda_fg_qto1}
    $\varphi_1(\Gamma_1)=\Lambda_1^G$. In particular, $\Lambda_1$ is finitely generated as a $\varphi_1(\Gamma_1)$-module.
\end{theorem}
\begin{proof}
By the definition $\varphi_1(a)=\rho_1\inv(\pi_1(a))$ in the proof of Theorem \ref{embedding_qto1}, and by Corollary \ref{gamma_contained_lambda_qto1}, we see
\begin{equation*}
    \varphi(C_n^{(2d)})=(-1)^d\sum_{1\leq p_1<p_2<\dots<p_d\leq k} \; \prod_{r=1}^d \big((x_{n,p_r}')^2-a_{p_r-r+1}\big).
\end{equation*}It is clear that each $(x_{nj}')^2$ is fixed under $\sigma_{nr}$ for any $r$. Also, $~^{\omega}(x_{nj}')^2=(x_{n,\omega(j)}')^2$ for all $\omega\in S_k$. It follows from \cite[Proposition~2.2]{capelli_id} that $\varphi(C_n^{(2d)})$ is symmetric in the $(x_{nj}')^2$ variables. When $n=2k$, $C_n^{(n)}$ is replaced by $C_n^{(n)+}$ as a generator. We have
\begin{equation*}
    \varphi(C_n^{(n)+})=(\sqrt{-1})^k\prod_{r=1}^k x_{nr}'. 
\end{equation*}Applying $\sigma_{nr}$ negates this expression for any $r$. By the definition of $W_n$ when $n$ is even, only an even number of sign changes can occur under the action of $W_n$, and of course this expression is symmetric in the $x_{nr}'$. Therefore $\varphi_1(\Gamma_1)\subseteq\Lambda_1^G$.

Before proving the reverse containment, we want to show that the elementary symmetric polynomial $e_d((x_{n1}')^2,(x_{n2}')^2,...,(x_{nk}')^2)\in\varphi_1(\Gamma_1^{(n)})$ for $1\leq d\leq k$. Replacing $(x_{nj}')^2$ with $b_{nj}$ from the proof of Theorem \ref{lambda_fg}, the argument here is analogous to the argument given in that proof.

Suppose $f\in(\Lambda_1^{(n)})^{W_n}$ is non-constant and $n=2k+1$. Since $f$ is fixed under the actions of all $\sigma_{nj}$ for $1\leq j\leq k$, it follows that $f\in\C[(x_{nj}')^{2}\mid 1\leq j\leq k]$. Consider a monomial $(x_{n1}')^{c_1}(x_{n2}')^{c_2}\dots (x_{nk}')^{c_k}$ in $f$ where $c_j$ is an even integer for all $1\leq j\leq k$. Since $f$ is fixed under the action of the symmetric group $S_k$, it follows that $f$ is symmetric with respect to the $(x_{nj}')^2$ variables. By Newton's Theorem, $f$ is generated by elementary symmetric polynomials $e_d((x_{n1}')^2,(x_{n2}')^2,...,(x_{nk}')^2)$, which are in $\varphi_1(\Gamma_1^{(n)})$.

Now suppose $f\in(\Lambda_1^{(n)})^{W_n}$ is non-constant and $n=2k$. $f$ is fixed under the action of every $\sigma_{nj_1}\sigma_{nj_2}$ for $1\leq j_1,j_2\leq k$. Thus for any monomial $(x_{n1}')^{c_1}(x_{n2}')^{c_2}\dots (x_{nk}')^{c_k}$ in $f$, the $c_j$ must be all even integers or all odd integers since $c_{j_1}+c_{j_2}\in 2\Z$. Split $f=f_{\text{even}}+f_{\text{odd}}$ where $f_{\text{even}}$ consists of the monomials where the exponents are even and $f_{\text{odd}}$ consists of the monomials where the exponents are odd. As in the previous paragraph, $f_{\text{even}}$ is generated by elementary symmetric polynomials $e_d((x_{n1}')^2,(x_{n2}')^2,...,(x_{nk}')^2)$, hence it is in $\varphi_1(\Gamma_1^{(n)})$. Now consider $f_{\text{odd}}$. Note that applying any $\sigma_{nj}$ to $f_{\text{odd}}$ yields $-f_{\text{odd}}$, so
\begin{equation*}
    f_{\text{odd}}(x_{n1}',...,x_{nk}')=\frac{1}{2}f_{\text{odd}}(x_{n1}',...,x_{nj}',...,x_{nk}')-\frac{1}{2}f_{\text{odd}}(x_{n1}',...,-x_{nj}',...,x_{nk}').
\end{equation*}It follows that $f_{\text{odd}}$ is zero when evaluated at $x_{nj}'=0$, so it is divisible by $x_{nj}'$. This holds for all $1\leq j\leq k$, so $f_{\text{odd}}$ is divisible by $g:=\prod_{j=1}^k x_{nj}'$. In particular, $f_{\text{odd}}=g\cdot g_{\text{even}}$ where $g_{\text{even}}$ is a polynomial such that the degree of every variable in each monomial is even. Note that $g$ is in $\varphi_1(\Gamma_1^{(n)})$ by $\varphi(C_n^{(n)+})$. For all $h\in W_n$, 
\begin{equation*}
    g\cdot g_{\text{even}}=f_{\text{odd}}=~^hf_{\text{odd}}=~^hg\cdot ~^h g_{\text{even}}=g\cdot ~^h g_{\text{even}},
\end{equation*}so $g_{\text{even}}$ is fixed by $W_n$ since $\Lambda_1^{(n)}$ is an integral domain. Following the argument that showed that $f_{\text{even}}\in \varphi_1(\Gamma_1^{(n)})$, we see that $g_{\text{even}}\in\varphi_1(\Gamma_1^{(n)})$.
\end{proof}

Recall the definition of a Harish-Chandra subalgebra in Definition \ref{hc_subalg_def} \cite{dfo}. The following lemma is proved by the same argument as in Lemma \ref{hc_subalg_smash_prod}.
\begin{lemma} \label{hc_subalg_smash_prod_qto1}
$\Lambda_1$ is a Harish-Chandra subalgebra of $L_1\#\m$.
\end{lemma}

The following is proved by the same argument as in Theorem \ref{hc_subalg}.
\begin{theorem} \label{hc_subalg_qto1}
    $\varphi_1(\Gamma_1)$ is a Harish-Chandra subalgebra of $L_1\#\m$. Furthermore, $\Gamma_1$ is a Harish-Chandra subalgebra of $\uson$.
\end{theorem}
\end{subsection}
\end{section}

\begin{section}{Appendix} \label{appendix}
In this section, we prove Lemma \ref{rat_telesc_prod}, a technical result which is used to prove Proposition \ref{rat}. We state the lemma again for the convenience of the reader. Recall that $e_i$ is the row-vector with $1$ for the $i$-th entry and where every other entry is zero.

\newtheorem*{lemmaapp}{Lemma \ref{rat_telesc_prod}}
\begin{lemmaapp} 
\begin{enumerate}[(a)]
    \item \label{rat_telesc_prod_odd_app}The following is true for $n=2p+1$.
    \begin{enumerate}[{\rm (i)}]
        \item \label{rat_telesc_prod_odd_app_i}When $1\leq j \leq p$,
        \begin{multline*}
            \frac{\lambda_n\begin{pmatrix}
    m_n\\
    m_{n-1}
    \end{pmatrix}}{\lambda_n\begin{pmatrix}
    m_n\\
    m_{n-1}+e_j
    \end{pmatrix}}=\bigg(\prod_{s=1}^{j-1}\frac{[l_s+l_j'][l_s'-l_j']}{[l_s'+l_j'+1][l_s-l_j'-1]}\bigg)^{1/2}\\
    \cdot \bigg(\frac{[l_j'-p+j][l_j'+1]}{[2l_j'-2p+2j][2l_j'+2]}\frac{[l_j+l_j'][l_j-l_j'-1]\prod_{r=j+1}^p[l_j'+l_r][l_j'-l_r+1]}{\prod_{r>j}[l_j'+l_r'+1][l_j'-l_r'+1]}\bigg)^{1/2}.
        \end{multline*}
        \item \label{rat_telesc_prod_odd_app_ii}When $1\leq j<p$, 
        \begin{multline*}
            \frac{\lambda_{n-1}\begin{pmatrix}
            m_{n-1}\\
            m_{n-2}
            \end{pmatrix}}{\lambda_{n-1}\begin{pmatrix}
            m_{n-1}+e_j\\
            m_{n-2}
            \end{pmatrix}}=\bigg(\prod_{s=1}^{j-1}\frac{[l_s''+l_j'][l_s'-l_j'-1]}{[l_s'+l_j'][l_s''-l_j'-1]}\bigg)^{1/2}\bigg(\frac{[l_j'+l_j'']}{[2l_j'][l_j'-l_j''+1]}\bigg)^{1/2}\\
    \cdot \bigg(\frac{[2l_j'-2p+2j][l_j']}{[l_j'-p+j]}\frac{\prod_{r>j}[l_j'+l_r''][l_j'-l_r''+1]}{\prod_{r=j+1}^p[l_j'+l_r'][l_j'-l_r']}\bigg)^{1/2}.
        \end{multline*}
        \item \label{rat_telesc_prod_odd_app_iii}When $j=p$, 
        \begin{equation*}
            \frac{\lambda_{n-1}\begin{pmatrix}
            m_{n-1}\\
            m_{n-2}
            \end{pmatrix}}{\lambda_{n-1}\begin{pmatrix}
            m_{n-1}+e_j\\
            m_{n-2}
            \end{pmatrix}}=\bigg(\prod_{s=1}^{p-1}\frac{[l_s''+l_p'][l_s'-l_p'-1]}{[l_s'+l_p'][l_s''-l_p'-1]}\bigg)^{1/2}.
        \end{equation*}
    \end{enumerate}
    \item \label{rat_telesc_prod_even_app}The following is true for $n=2p$.
    \begin{enumerate}[{\rm (i)}]
        \item \label{rat_telesc_prod_even_app_i}When $1\leq j\leq p-1$,
        \begin{multline*}
            \frac{\lambda_n\begin{pmatrix}
    m_n\\
    m_{n-1}
    \end{pmatrix}}{\lambda_n\begin{pmatrix}
    m_n\\
    m_{n-1}+e_j
    \end{pmatrix}}=
    \bigg(\prod_{s=1}^{j-1}\frac{[l_s'+l_j''][l_s''-l_j'']}{[l_s''+l_j''][l_s'-l_j'']}\bigg)^{1/2}\\
    \cdot \bigg(\frac{[l_j''-p+j]}{[2l_j''-2p+2j]}\frac{[l_j'+l_j''][l_j'-l_j'']\prod_{r=j+1}^{p}[l_j''+l_r'][l_j''-l_r']}{\prod_{r>j}[l_j''+l_r''][l_j''-l_r''+1]}\frac{1}{[l_j''][2l_j''+1]}\bigg)^{1/2}.
        \end{multline*}
        \item \label{rat_telesc_prod_even_app_ii}When $1\leq j\leq p-1$, 
        \begin{multline*}
            \frac{\lambda_{n-1}\begin{pmatrix}
            m_{n-1}\\
            m_{n-2}
            \end{pmatrix}}{\lambda_{n-1}\begin{pmatrix}
            m_{n-1}+e_j\\
            m_{n-2}
            \end{pmatrix}}=\bigg(\prod_{s=1}^{j-1}\frac{[l_s'''+l_j''][l_s''-l_j''-1]}{[l_s''+l_j''-1][l_s'''-l_j'']}\bigg)^{1/2}\bigg(\frac{[l_j''+l_j''']}{[2l_j''-1][l_j''-l_j''']}\bigg)^{1/2}\\
    \cdot \bigg(\frac{[2l_j''-2p+2j]}{[l_j''-p+j][l_j'']}\frac{\prod_{r>j}[l_j''+l_r'''][l_j''-l_r''']}{\prod_{r=j+1}^{p-1}[l_j''+l_r''-1][l_j''-l_r'']}\bigg)^{1/2}.
        \end{multline*}
    \end{enumerate}
\end{enumerate}
\end{lemmaapp}

\begin{proof}[Proof of Lemma \ref{rat_telesc_prod}]
By the definition of our recursion for the re-scaling, setting $p_{n-1}:=\lfloor\frac{n-1}{2}\rfloor$ we have
\begin{equation} \label{rescale_lambda_n}
    \lambda_n\begin{pmatrix}
    m_n\\
    m_{n-1}
    \end{pmatrix}=\prod_{s=1}^{p_{n-1}}\prod_{k=1}^{m_{ns}-m_{n-1,s}}\frac{\lambda_n\begin{pmatrix}
    m_n\\
    m_{n-1}+\sum_{i=1}^{s-1}(m_{ni}-m_{n-1,i})e_i+(k-1)e_s
    \end{pmatrix}}{\lambda_n\begin{pmatrix}
    m_n\\
    m_{n-1}+\sum_{i=1}^{s-1}(m_{ni}-m_{n-1,i})e_i+k\cdot e_s
    \end{pmatrix}}
\end{equation}since it is a telescoping product and  
\begin{equation*}
\lambda_n\begin{pmatrix}
    m_n\\
    m_{n-1}+\sum_{i=1}^{p_{n-1}}(m_{ni}-m_{n-1,i})e_i
    \end{pmatrix}:=1.
    \end{equation*}

\noindent Then for items (a)(i) and (b)(i), 
\begin{multline*}
    \frac{\lambda_n\begin{pmatrix}
    m_n\\
    m_{n-1}
    \end{pmatrix}}{\lambda_n\begin{pmatrix}
    m_n\\
    m_{n-1}+e_j
    \end{pmatrix}}=\prod_{s=1}^{j-1}\Bigg\{\prod_{k=1}^{m_{ns}-m_{n-1,s}}\frac{\lambda_n\begin{pmatrix}
    m_n\\
    m_{n-1}+\sum_{i=1}^{s-1}(m_{ni}-m_{n-1,i})e_i+(k-1)e_s
    \end{pmatrix}}{\lambda_n\begin{pmatrix}
    m_n\\
    m_{n-1}+\sum_{i=1}^{s-1}(m_{ni}-m_{n-1,i})e_i+k\cdot e_s
    \end{pmatrix}}\\
    \cdot \Bigg(\frac{\lambda_n\begin{pmatrix}
    m_n\\
    m_{n-1}+\sum_{i=1}^{s-1}(m_{ni}-m_{n-1,i})e_i+(k-1)e_s+e_j
    \end{pmatrix}}{\lambda_n\begin{pmatrix}
    m_n\\
    m_{n-1}+\sum_{i=1}^{s-1}(m_{ni}-m_{n-1,i})e_i+k\cdot e_s+e_j
    \end{pmatrix}}\Bigg)^{-1}\Bigg\} \\
    \cdot \frac{\lambda_n\begin{pmatrix}
    m_n\\
    m_{n-1}+\sum_{i=1}^{j-1}(m_{ni}-m_{n-1,i})e_i
    \end{pmatrix}}{\lambda_n\begin{pmatrix}
    m_n\\
    m_{n-1}+\sum_{i=1}^{j-1}(m_{ni}-m_{n-1,i})e_i+e_j
    \end{pmatrix}}.
\end{multline*}The middle line of this equation comes from the denominator, and the reason we do not see more terms for greater values of $s$ is that they are cancelling from the quotient. Then for item (a)(i) we have
\begin{multline*}
    \frac{\lambda_n\begin{pmatrix}
    m_n\\
    m_{n-1}
    \end{pmatrix}}{\lambda_n\begin{pmatrix}
    m_n\\
    m_{n-1}+e_j
    \end{pmatrix}}= \bigg(\prod_{s=1}^{j-1}\prod_{k=1}^{l_s-l_s'-1}\frac{[(l_s'+k-1)+(l_j'+1)+1][(l_s'+k-1)-(l_j'+1)+1]}{[(l_s'+k-1)+l_j'+1][(l_s'+k-1)-l_j'+1]}\bigg)^{1/2}\\
    \cdot \frac{\lambda_n\begin{pmatrix}
    m_n\\
    m_{n-1}+\sum_{i=1}^{j-1}(m_{ni}-m_{n-1,i})e_i
    \end{pmatrix}}{\lambda_n\begin{pmatrix}
    m_n\\
    m_{n-1}+\sum_{i=1}^{j-1}(m_{ni}-m_{n-1,i})e_i+e_j
    \end{pmatrix}}
\end{multline*}
\begin{multline*}
    =\bigg(\prod_{s=1}^{j-1}\frac{[l_s+l_j'][l_s'-l_j']}{[l_s'+l_j'+1][l_s-l_j'-1]}\bigg)^{1/2}\\
    \cdot \bigg(\frac{[l_j'-p+j][l_j'+1]}{[2l_j'-2p+2j][2l_j'+2]}\frac{[l_j+l_j'][l_j-l_j'-1]\prod_{r=j+1}^p[l_j'+l_r][l_j'-l_r+1]}{\prod_{r>j}[l_j'+l_r'+1][l_j'-l_r'+1]}\bigg)^{1/2},
\end{multline*}and for item (b)(i) we have
\begin{multline*}
    \frac{\lambda_n\begin{pmatrix}
    m_n\\
    m_{n-1}
    \end{pmatrix}}{\lambda_n\begin{pmatrix}
    m_n\\
    m_{n-1}+e_j
    \end{pmatrix}}=\bigg(\prod_{s=1}^{j-1}\prod_{k=1}^{l_s'-l_s''}\frac{[(l_s''+k-1)+(l_j''+1)][(l_s''+k-1)-(l_j''+1)+1]}{[(l_s''+k-1)+l_j''][(l_s''+k-1)-l_j''+1]}\bigg)^{1/2}\\
    \cdot \frac{\lambda_n\begin{pmatrix}
    m_n\\
    m_{n-1}+\sum_{i=1}^{j-1}(m_{ni}-m_{n-1,i})e_i
    \end{pmatrix}}{\lambda_n\begin{pmatrix}
    m_n\\
    m_{n-1}+\sum_{i=1}^{j-1}(m_{ni}-m_{n-1,i})e_i+e_j
    \end{pmatrix}}
    \end{multline*}
    \begin{multline*}
    =\bigg(\prod_{s=1}^{j-1}\frac{[l_s'+l_j''][l_s''-l_j'']}{[l_s''+l_j''][l_s'-l_j'']}\bigg)^{1/2}\\
    \cdot \bigg(\frac{[l_j''-p+j]}{[2l_j''-2p+2j]}\frac{[l_j'+l_j''][l_j'-l_j'']\prod_{r=j+1}^{p}[l_j''+l_r'][l_j''-l_r']}{\prod_{r>j}[l_j''+l_r''][l_j''-l_r''+1]}\frac{1}{[l_j''][2l_j''+1]}\bigg)^{1/2}.
\end{multline*}

\noindent Now for item (a)(ii) we have
\begin{multline*}
    \frac{\lambda_{n-1}\begin{pmatrix}
    m_{n-1}\\
    m_{n-2}
    \end{pmatrix}}{\lambda_{n-1}\begin{pmatrix}
    m_{n-1}+e_j\\
    m_{n-2}
    \end{pmatrix}}=\prod_{s=1}^j\Bigg\{\prod_{k=1}^{m_{n-1,s}-m_{n-2,s}}\frac{\lambda_{n-1}\begin{pmatrix}
    m_{n-1}\\
    m_{n-2}+\sum_{i=1}^{s-1}(m_{n-1,i}-m_{n-2,i})e_i+(k-1)e_s
    \end{pmatrix}}{\lambda_{n-1}\begin{pmatrix}
    m_{n-1}\\
    m_{n-2}+\sum_{i=1}^{s-1}(m_{n-1,i}-m_{n-2,i})e_i+k\cdot e_s
    \end{pmatrix}}\\
    \cdot \Bigg(\frac{\lambda_{n-1}\begin{pmatrix}
    m_{n-1}+e_j\\
    m_{n-2}+\sum_{i=1}^{s-1}(m_{n-1,i}-m_{n-2,i})e_i+(k-1)e_s
    \end{pmatrix}}{\lambda_{n-1}\begin{pmatrix}
    m_{n-1}+e_j\\
    m_{n-2}+\sum_{i=1}^{s-1}(m_{n-1,i}-m_{n-2,i})e_i+k\cdot e_s
    \end{pmatrix}}\Bigg)^{-1}\Bigg\}\\
    \cdot \Bigg(\frac{\lambda_{n-1}\begin{pmatrix}
    m_{n-1}+e_j\\
    m_{n-2}+\sum_{i=1}^{j}(m_{n-1,i}-m_{n-2,i})e_i
    \end{pmatrix}}{\lambda_{n-1}\begin{pmatrix}
    m_{n-1}+e_j\\
    m_{n-2}+\sum_{i=1}^{j}(m_{n-1,i}-m_{n-2,i})e_i+e_j
    \end{pmatrix}}\Bigg)^{-1}\\
    \cdot \prod_{s=j+1}^{p-1}\Bigg\{\prod_{k=1}^{m_{n-1,s}-m_{n-2,s}}\frac{\lambda_{n-1}\begin{pmatrix}
    m_{n-1}\\
    m_{n-2}+\sum_{i=1}^{s-1}(m_{n-1,i}-m_{n-2,i})e_i+(k-1)e_s
    \end{pmatrix}}{\lambda_{n-1}\begin{pmatrix}
    m_{n-1}\\
    m_{n-2}+\sum_{i=1}^{s-1}(m_{n-1,i}-m_{n-2,i})e_i+k\cdot e_s
    \end{pmatrix}}\\
    \cdot \Bigg(\frac{\lambda_{n-1}\begin{pmatrix}
    m_{n-1}+e_j\\
    m_{n-2}+\sum_{i=1}^{s-1}(m_{n-1,i}-m_{n-2,i})e_i+e_j+(k-1)e_s
    \end{pmatrix}}{\lambda_{n-1}\begin{pmatrix}
    m_{n-1}+e_j\\
    m_{n-2}+\sum_{i=1}^{s-1}(m_{n-1,i}-m_{n-2,i})e_i+e_j+k\cdot e_s
    \end{pmatrix}}\Bigg)^{-1}\Bigg\}.
\end{multline*}This follows from \eqref{rat_recursion_even}. The terms with the $-1$ superscript are from the denominator. Note that the double product at the bottom consists of terms which do not depend on entries from the first $j$ columns in the GT pattern, so this double product is equal to 1 by cancellation. For item (a)(ii), the third line in the equation is the reciprocal of
\begin{equation*}
    \bigg(\frac{[l_j'-p+j]}{[2l_j'-2p+2j]}\frac{[2l_j'+1][1]}{[l_j'][2l_j'+1]}\frac{\prod_{r=j+1}^p[l_j'+l_r'][l_j'-l_r']}{\prod_{r>j}[l_j'+l_r''][l_j'-l_r''+1]}\bigg)^{1/2}.
\end{equation*}This follows from \eqref{rat_recursion_even} but we replace $l_{j}'$ with $l_{j}'+1$ and $l_{j}''$ with $l_{j}'$. (We replace $m_{n-2,j}=l_j''-(p-1-j+1)$ with $m_{n-1,j}=l_j'-(p-j)$.) Then for item (a)(ii) we have
\begin{multline*}
     \frac{\lambda_{n-1}\begin{pmatrix}
    m_{n-1}\\
    m_{n-2}
    \end{pmatrix}}{\lambda_{n-1}\begin{pmatrix}
    m_{n-1}+e_j\\
    m_{n-2}
    \end{pmatrix}}=\bigg(\prod_{s=1}^{j-1}\prod_{k=1}^{l_s'-l_s''}\frac{[(l_s''+k-1)+l_j'][(l_s''+k-1)-l_j']}{[(l_s''+k-1)+(l_j'+1)][(l_s''+k-1)-(l_j'+1)]}\bigg)^{1/2}\\
    \cdot \bigg(\prod_{k=1}^{l_j'-l_j''}\frac{[l_j'+(l_j''+k-1)][l_j'-(l_j''+k-1)]}{[(l_j'+1)+(l_j''+k-1)][(l_j'+1)-(l_j''+k-1)]}\bigg)^{1/2}\\
    \cdot \bigg(\frac{[l_j'-p+j]}{[2l_j'-2p+2j]}\frac{[2l_j'+1][1]}{[l_j'][2l_j'+1]}\frac{\prod_{r=j+1}^p[l_j'+l_r'][l_j'-l_r']}{\prod_{r>j}[l_j'+l_r''][l_j'-l_r''+1]}\bigg)^{-1/2}
\end{multline*}
\begin{multline*}
    =\bigg(\prod_{s=1}^{j-1}\frac{[l_s''+l_j'][l_s'-l_j'-1]}{[l_s'+l_j'][l_s''-l_j'-1]}\bigg)^{1/2}\bigg(\frac{[l_j'+l_j'']}{[2l_j'][l_j'-l_j''+1]}\bigg)^{1/2}\\
    \cdot \bigg(\frac{[2l_j'-2p+2j][l_j']}{[l_j'-p+j]}\frac{\prod_{r>j}[l_j'+l_r''][l_j'-l_r''+1]}{\prod_{r=j+1}^p[l_j'+l_r'][l_j'-l_r']}\bigg)^{1/2}.
\end{multline*}

\noindent Similarly to item (a)(ii), for item (b)(ii) we have
\begin{multline*}
    \frac{\lambda_{n-1}\begin{pmatrix}
    m_{n-1}\\
    m_{n-2}
    \end{pmatrix}}{\lambda_{n-1}\begin{pmatrix}
    m_{n-1}+e_j\\
    m_{n-2}
    \end{pmatrix}}=\prod_{s=1}^j\Bigg\{\prod_{k=1}^{m_{n-1,s}-m_{n-2,s}}\frac{\lambda_{n-1}\begin{pmatrix}
    m_{n-1}\\
    m_{n-2}+\sum_{i=1}^{s-1}(m_{n-1,i}-m_{n-2,i})e_i+(k-1)e_s
    \end{pmatrix}}{\lambda_{n-1}\begin{pmatrix}
    m_{n-1}\\
    m_{n-2}+\sum_{i=1}^{s-1}(m_{n-1,i}-m_{n-2,i})e_i+k\cdot e_s
    \end{pmatrix}}\\
    \cdot \Bigg(\frac{\lambda_{n-1}\begin{pmatrix}
    m_{n-1}+e_j\\
    m_{n-2}+\sum_{i=1}^{s-1}(m_{n-1,i}-m_{n-2,i})e_i+(k-1)e_s
    \end{pmatrix}}{\lambda_{n-1}\begin{pmatrix}
    m_{n-1}+e_j\\
    m_{n-2}+\sum_{i=1}^{s-1}(m_{n-1,i}-m_{n-2,i})e_i+k\cdot e_s
    \end{pmatrix}}\Bigg)^{-1}\Bigg\}\\
    \cdot \Bigg(\frac{\lambda_{n-1}\begin{pmatrix}
    m_{n-1}+e_j\\
    m_{n-2}+\sum_{i=1}^{j}(m_{n-1,i}-m_{n-2,i})e_i
    \end{pmatrix}}{\lambda_{n-1}\begin{pmatrix}
    m_{n-1}+e_j\\
    m_{n-2}+\sum_{i=1}^{j}(m_{n-1,i}-m_{n-2,i})e_i+e_j
    \end{pmatrix}}\Bigg)^{-1}\\
    \cdot \prod_{s=j+1}^{p-1}\Bigg\{\prod_{k=1}^{m_{n-1,s}-m_{n-2,s}}\frac{\lambda_{n-1}\begin{pmatrix}
    m_{n-1}\\
    m_{n-2}+\sum_{i=1}^{s-1}(m_{n-1,i}-m_{n-2,i})e_i+(k-1)e_s
    \end{pmatrix}}{\lambda_{n-1}\begin{pmatrix}
    m_{n-1}\\
    m_{n-2}+\sum_{i=1}^{s-1}(m_{n-1,i}-m_{n-2,i})e_i+k\cdot e_s
    \end{pmatrix}}\\
    \cdot \Bigg(\frac{\lambda_{n-1}\begin{pmatrix}
    m_{n-1}+e_j\\
    m_{n-2}+\sum_{i=1}^{s-1}(m_{n-1,i}-m_{n-2,i})e_i+e_j+(k-1)e_s
    \end{pmatrix}}{\lambda_{n-1}\begin{pmatrix}
    m_{n-1}+e_j\\
    m_{n-2}+\sum_{i=1}^{s-1}(m_{n-1,i}-m_{n-2,i})e_i+e_j+k\cdot e_s
    \end{pmatrix}}\Bigg)^{-1}\Bigg\}.
\end{multline*}This follows from \eqref{rat_recursion_odd}. Again, every term with the superscript $-1$ comes from the denominator, and the double product at the bottom is equal to $1$. The third line is the reciprocal of
\begin{equation*}
 \bigg(\frac{[l_j''-p+j][l_j'']}{[2l_j''-2p+2j][2l_j'']}\frac{[2l_j''][1]\prod_{r=j+1}^{p-1}[l_j''+l_r''-1][l_j''-l_r'']}{\prod_{r>j}[l_j''+l_r'''][l_j''-l_r''']}\bigg)^{1/2}.
\end{equation*}This follows from \eqref{rat_recursion_odd} but first we replace $p$ with $p-1$, $l_r$ with $l_r''$, and $l_r'$ with $l_r'''$ for all $r\geq j$ since $n=2p+1$ in \eqref{rat_recursion_odd} but we must replace $n$ with $2p-1$ here. Then we replace $l_j''$ with $l_j''+1$ and $l_j'''$ with $l_j''-1$. (We replace $m_{n-2,j}=l_j'''-(p-1-j)$ with $m_{n-1,j}=l_j''-(p-1-j+1)$.) Then
\begin{multline*}
    \frac{\lambda_{n-1}\begin{pmatrix}
    m_{n-1}\\
    m_{n-2}
    \end{pmatrix}}{\lambda_{n-1}\begin{pmatrix}
    m_{n-1}+e_j\\
    m_{n-2}
    \end{pmatrix}}=\bigg(\prod_{s=1}^{j-1}\prod_{k=1}^{l_s''-l_s'''-1}\frac{[(l_s'''+k-1)+l_j''][(l_s'''+k-1)-l_j''+1]}{[(l_s'''+k-1)+(l_j''+1)][(l_s'''+k-1)-(l_j''+1)+1]} \bigg)^{1/2}\\
    \cdot \bigg(\prod_{k=1}^{l_j''-l_j'''-1}\frac{[l_j''+(l_j'''+k-1)][l_j''-(l_j'''+k-1)-1]}{[(l_j''+1)+(l_j'''+k-1)][(l_j''+1)-(l_j'''+k-1)-1]} \bigg)^{1/2}\\
    \cdot \bigg(\frac{[l_j''-p+j][l_j'']}{[2l_j''-2p+2j][2l_j'']}\frac{[2l_j''][1]\prod_{r=j+1}^{p-1}[l_j''+l_r''-1][l_j''-l_r'']}{\prod_{r>j}[l_j''+l_r'''][l_j''-l_r''']}\bigg)^{-1/2}
\end{multline*}
\begin{multline*}
    =\bigg(\prod_{s=1}^{j-1}\frac{[l_s'''+l_j''][l_s''-l_j''-1]}{[l_s''+l_j''-1][l_s'''-l_j'']}\bigg)^{1/2}\bigg(\frac{[l_j''+l_j''']}{[2l_j''-1][l_j''-l_j''']}\bigg)^{1/2}\\
    \cdot \bigg(\frac{[2l_j''-2p+2j]}{[l_j''-p+j][l_j'']}\frac{\prod_{r>j}[l_j''+l_r'''][l_j''-l_r''']}{\prod_{r=j+1}^{p-1}[l_j''+l_r''-1][l_j''-l_r'']}\bigg)^{1/2}.
\end{multline*}

\noindent Lastly, for item (a)(iii) we have
\begin{multline*}
    \frac{\lambda_{n-1}\begin{pmatrix}
    m_{n-1}\\
    m_{n-2}
    \end{pmatrix}}{\lambda_{n-1}\begin{pmatrix}
    m_{n-1}+e_j\\
    m_{n-2}
    \end{pmatrix}}=\prod_{s=1}^{p-1}\Bigg\{\prod_{k=1}^{m_{n-1,s}-m_{n-2,s}}\frac{\lambda_{n-1}\begin{pmatrix}
    m_{n-1}\\
    m_{n-2}+\sum_{i=1}^{s-1}(m_{n-1,i}-m_{n-2,i})e_i+(k-1)e_s
    \end{pmatrix}}{\lambda_{n-1}\begin{pmatrix}
    m_{n-1}\\
    m_{n-2}+\sum_{i=1}^{s-1}(m_{n-1,i}-m_{n-2,i})e_i+k\cdot e_s
    \end{pmatrix}}\\
    \cdot \Bigg(\frac{\lambda_{n-1}\begin{pmatrix}
    m_{n-1}+e_j\\
    m_{n-2}+\sum_{i=1}^{s-1}(m_{n-1,i}-m_{n-2,i})e_i+(k-1)e_s
    \end{pmatrix}}{\lambda_{n-1}\begin{pmatrix}
    m_{n-1}+e_j\\
    m_{n-2}+\sum_{i=1}^{s-1}(m_{n-1,i}-m_{n-2,i})e_i+k\cdot e_s
    \end{pmatrix}}\Bigg)^{-1}\Bigg\}
\end{multline*}
\begin{equation*}
    =\bigg(\prod_{s=1}^{p-1}\frac{[l_s''+l_p'][l_s'-l_p'-1]}{[l_s'+l_p'][l_s''-l_p'-1]}\bigg)^{1/2}.
\end{equation*}
\end{proof}
\end{section}

\end{document}